\documentclass{amsart} %

\usepackage{dsfont} %
\usepackage{xspace} %
\usepackage[pdftex]{graphicx} %

\usepackage{amssymb}

\def\setF{\mathbb{F}} %
\def\setZ{\mathbb{Z}} %
\def\setR{\mathbb{R}} %
\def\setN{\mathbb{N}} %
\def\setC{\mathbb{C}} %

\def\catA{\mathcal{A}} %
\def\catB{\mathcal{B}} %
\def\catN{\mathcal{N}} %
\def\catP{\mathcal{P}} %
\def\catQ{\mathcal{Q}} %

\providecommand{\ds}{\displaystyle} %

\DeclareMathOperator{\End}{End} %

\newtheorem{lemma}{Lemma}[section] %
\newtheorem{proposition}[lemma]{Proposition} %
\newtheorem{theorem}[lemma]{Theorem} %
\newtheorem{corollary}[lemma]{Corollary} %

\theoremstyle{remark} %

\theoremstyle{definition} %
\newtheorem{definition}[lemma]{Definition} %

\newcommand{\nil}{\mathcal N} %
\newcommand{\vnil}{\widetilde V \times \nil} %
\newcommand{\kvnil}{K \backslash (\widetilde V \times \mathcal{N})} %
\newcommand{\knil}{K \backslash \mathcal{N}} %
\newcommand{\orb}{\mathcal O} %
\newcommand{\sub}[1]{\ensuremath{\{1, \dots, #1 \}}} %
\newcommand{\rep}[1]{\langle #1 \rangle} %

\newcommand{\floor}[1]{\left\lfloor #1 \right\rfloor} %
\newcommand{\ceil}[1]{\left\lceil #1 \right\rceil} %

\newcommand{\n}{\ensuremath{n}} %

\newcommand{\zmod}[1][\n]{\setZ / #1 \setZ} %
\newcommand{\nneg}{\setZ^{\geq 0}} %

\newcommand{\qb}[1][\n]{\ensuremath{#1}-bipartition\xspace} %
\newcommand{\cqb}[1][\n]{striped \ensuremath{#1}-bipartition\xspace} %
\newcommand{\gqb}[1][\n]{generalized striped \ensuremath{#1}-bipartition\xspace} %

\newcommand{\Span}[1]{\left\langle#1\right\rangle} %
\newcommand{\g}{\mathfrak g}

\title{Enhanced Nilpotent Representations of a Cyclic Quiver} %
\author{Casey P.  Johnson} %
\date{20 Apr 2010} %

\begin{document}
\begin{abstract}
  We define a set of ``enhanced'' nilpotent quiver representations
  that generalizes the enhanced nilpotent cone.  This set admits an
  action by an associated algebraic group $K$ with finitely many
  orbits.  We define a combinatorial set that parametrizes the set of
  orbits under this action and we derive a purely combinatorial
  formula for the dimension of an orbit.
\end{abstract}
\maketitle

\section{Introduction} \label{sec:intro}

\subsection{The enhanced nilpotent cone}

In his study of the exotic Springer correspondence in \cite{kato-2006}
and the exotic Deligne-Langlands correspondence in \cite{kato}, Kato
introduces an object that he calls the \emph{exotic nilpotent cone}.
If $U$ is a $2 k$-dimensional symplectic vector space, let $N_0$
denote the set of nilpotent self-adjoint endomorphisms of $U$.  The
exotic nilpotent cone is the set $U \times N_0$ and it admits a
natural action by $K = Sp(U)$.

It has long been known that if $W$ is the Weyl group of type $C_k$
then the set $\widehat W$ of equivalence classes of irreducible
representations of $W$ is in bijection with the set of pairs $(\mu;
\nu)$ of partitions such that $|\mu| + |\nu| = k$.  Kato showed that
this set of ``bipartitions of size $k$'' is also naturally in
bijection with $K \setminus (U \times N_0)$, the set of orbits of $K$
on $U \times N_0$, which gives an alternative parametrization of
$\widehat W$ by $K \setminus (U \times N_0)$.

There are two \emph{enhanced nilpotent cones} closely associated to
the exotic nilpotent cone.  If $V$ is a linear space and $\nil(V)$
denotes the set of nilpotent linear endomorphisms of $V$ then the
enhanced nilpotent cone of $V$ is the set $V \times \nil(V)$.  It is
easy to see that if $V$ is a Lagrangian subspace of $U$ then $V \times
\nil(V) \subset U \times N_0 \subset U \times \nil(U)$.  On each of
these varieties there is a natural group action, namely
\begin{itemize}
\item $GL(V)$ acts on $V \times \nil(V)$,
\item $Sp(U)$ acts on $U \times N_0$,
\item $GL(U)$ acts on $U \times \nil(U)$.
\end{itemize}

Travkin proves in \cite{travkin} that $GL(V) \setminus (V \times
\nil(V))$ is parametrized by the set of bipartitions of size $k$, so
$GL(U) \setminus (U \times \nil(U))$ is parametrized by the set of
bipartitions of size $2 k$.  Achar and Henderson independently prove
the same result in \cite{AH}, going on to show that there is a natural
embedding $GL(V) \subset Sp(U) \subset GL(U)$ and that these three
parametrizations have the important compatibility property given
below.  In the statement that follows, let $\mu \cup \mu$ denote the
partition of size $2 k$ obtained from $\mu$ by doubling the
multiplicity of each row.

\begin{theorem} \label{thm:compatible} (Achar-Henderson) If $(\mu;
  \nu)$ is a bipartition and $\orb_{\mu; \nu}$ and $\mathbb O_{\mu;
    \nu}$ denote the corresponding enhanced and exotic orbits,
  respectively, then $\orb_{\mu; \nu} \subset \mathbb O_{\mu; \nu}
  \subset \orb_{\mu \cup \mu; \nu \cup \nu}$.
\end{theorem}

Since each of these actions yields finitely many orbits and the groups
acting are algebraic, we have the natural partial order on orbits
defined by closure.  That is, we can say that $\orb_{\mu; \nu} \leq
\orb_{\mu'; \nu'}$ if and only if $\orb_{\mu; \nu}$ is contained in
the Zariski closure of $\orb_{\mu'; \nu'}$.  Achar and Henderson
define a combinatorial partial order $\leq$ on the set of bipartitions
of size $k$ and prove the following.

\begin{theorem}
  (Achar-Henderson) The following are equivalent:
  \begin{enumerate}
  \item $(\mu; \nu) \leq (\mu'; \nu')$
  \item $\orb_{\mu; \nu} \subset \overline \orb_{\mu'; \nu'}$
  \item $\mathbb O_{\mu; \nu} \subset \overline {\mathbb O}_{\mu';
      \nu'}$
  \end{enumerate}
\end{theorem}

Henderson has proved in \cite{henderson} that, for each $\lambda$,
$\overline {\mathbb O}_{\emptyset; \nu}$ has the same intersection
cohomology as $\overline \orb_{\emptyset; \nu}$, with all degrees
doubled.  He and Achar conjecture in \cite{AH} that the same holds for
all bipartitions $(\mu; \nu)$ and they also outline a programme for
investigating this conjecture.

\subsection{Nilpotent cyclic quiver representations}

Achar-Henderson's parametrization begins with the well-known fact that
if $V$ is a finite-dimensional linear space then the Jordan normal
form parametrizes the conjugacy classes of nilpotent matrices.  Since
the Jordan form of a nilpotent matrix corresponds to a partition of
size $k = \dim V$, there is a natural bijection \[\{\text{partitions
  of size } k \} \longleftrightarrow \{\text{conjugacy classes in }
\nil(V)\}.\] Furthermore, $\nil(V)$ embeds in $V \times \nil(V)$ as
$\{0\} \times \nil(V)$ and the set of partitions embeds in the set of
bipartitions via $\nu \mapsto (\emptyset; \nu)$ in such a way
$\orb_{\nu} \cong \orb_{\emptyset; \nu}$.  In other words, the
parameter set reduces to the classical parametrization when the
enhanced nilpotent orbits are just ordinary nilpotent orbits in
disguise.

On the other hand, we can generalize the nilpotent cone in another
way.  Let $\Gamma$ be a cyclic quiver of order $\n$.  We can view
$\Gamma$ as the set $X = \zmod$ with directed edges $e_i = (i, i +
[1]), i \in X$.  A representation of $\Gamma$ assigns to each $i \in
X$ a finite-dimensional vector space $V_i$ and a linear transformation
$x_i \in \operatorname{Hom} (V_i, V_{i + [1]})$.  We say that such a
representation is nilpotent if $x_{[\n - 1]} \circ \cdots \circ
x_{[1]} \circ x_{[0]} \in \End(V_{[0]})$ is nilpotent.

If we fix $V_i$ for each $i \in X$, we can consider the set $\nil$ of
nilpotent quiver representations of $\Gamma$ with the chosen
underlying vector spaces.  Then $K = \prod_{i \in X} GL(V_i)$
naturally acts on $V = \prod_{i \in X} V_i$, hence on $\nil$ by
conjugation.  Thus, we can consider the problem of parametrizing the
set $K \backslash \nil$ of orbits of this action.  Kempken solves this
problem in \cite{kempken} for the case of a cyclic graph, showing that
these orbits are parametrized by a generalization of the classical
notion of partition, which we will call ``colored partitions.''  In
addition, Kempken presents a combinatorial description of the closure
order in $K \backslash \nil$.  In sections \ref{sec:colored} and
\ref{sec:nil} we present a full exposition of the parametrization,
culminating in theorem \ref{thm:paramc}.

The case where $\Gamma$ is a $2$-cycle is of particular interest.  If
$G$ is the real Lie group $U(p, q)$, with Lie algebra $\g = \mathfrak
u(p,q)$, then the set of nilpotent adjoint orbits in $\g$ is
parametrized by the set of signed (2-colored) partitions of signature
$(p,q)$ in, e.g., \cite{collingwood}.  On the other hand, if $K =
GL(p, \setC) \times GL(q, \setC)$ and $\catN = \{(x, y) \mid x:
\setC^p \to \setC^q, y: \setC^q \to \setC^p \text{ are linear with } x
\circ y \text{ nilpotent}\}$ then the Kostant-Sekiguchi bijection is a
natural one-to-one correspondence between the set of nilpotent adjoint
orbits and $K \backslash \nil$.  Thus, we can view the set of adjoint
orbits as a set of orbits of quiver representations over a cyclic
graph of order 2.

\subsection{Main results}

The objective of this paper is to present a framework that generalizes
both of these constructions.  We ``enhance'' the set of nilpotent
quiver representations of a cyclic graph by taking its product with
the natural representation $V_i$ of $K$, for some $i \in X$.  $K$
naturally acts on $V_i \times \nil$ with finitely many orbits.  In
fact, we will take the product of $\nil$ with the slightly larger
space $\widetilde V = \bigcup_{i \in X} V_i$ that includes $V_i$ for
each $i \in X$.

In theorem \ref{thm:param} we show that the set $\kvnil$ of orbits is
finite and is parametrized by the set of ``\cqb{}s'' defined in
section \ref{ssec:marked}.  Essentially, a \cqb is a partition that is
colored to reflect the quiver structure and also divided in two parts,
each of which is a natural deformation of a partition.  As a
consequence, we obtain a parametrization of $K \backslash (V_i \times
\nil)$.

In the case $\n = 1$ the set of \cqb{}s reduces precisely to the set
of bipartitions, yielding the Achar-Henderson parametrization.  On the
other hand, we have the natural embedding $\{0\} \times \nil \subset
\vnil$ and we will show that the parameters that correspond to orbits
in $\{0\} \times \nil$ can be viewed as colored partitions in a
natural way that reduces to the usual parametrization of $\knil$.

Lastly, we derive formulas for computing the dimension of an orbit
given its corresponding \cqb.  These formulas quickly reduce to the
formulas that have been given by Achar-Henderson and Kempken.  We are
particularly interested in the case $\n = 2$ discussed above.  In this
setting, the \cqb{}s yield especially simple dimension formulas, which
are included as corollaries \ref{cor:dimA}, \ref{cor:dimB}, and
\ref{cor:dim}. With this framework in place, we will be in a position
to explore the closure order---a topic that will be covered in a
future paper.

\section{Colored vector spaces} \label{sec:colored}

Most of the constructions in this paper rely on the notion of a
colored vector space.  In this section we introduce colored vector
spaces and we develop their basic structure, including a few
properties of their automorphisms and endomorphisms.  This section is
elementary in nature, so few proofs are included.  In most cases, the
claims are explicit enough to suggest a proof.

\subsection{Notation}

Throughout this paper we fix the following notational conventions,
most of which are standard.

\begin{enumerate}
\item $\setZ$ is the additive group of integers and $\nneg$ is the set
  of nonnegative integers.
\item $\setN$ is the additive semigroup of positive integers.
\item $\setR$ and $\setC$ are the fields of real and complex numbers,
  respectively.
\item $\n$ is a fixed positive integer.
\item If $k$ is an integer then $\zmod[k]$ is the usual quotient
  group, the cyclic group with $k$ elements.  If $i \in \setZ$ then we
  write $[i] = i + k \setZ \in \zmod[k]$.  To prevent notational
  clutter, if $0 \leq i < k$ then we will write $i$ rather than $[i]$
  whenever we can do so unambiguously.  If we need to be more explicit
  in choosing a particular representative of $[i]$, we will write
  $\rep i$ or $\rep{[i]}$ to denote the smallest nonnegative element
  of $[i]$.
\item $\floor{\cdot}$ is the floor function: $\floor x = \max\{y \in
  \setZ \mid y \leq x\}$.
\item $\ceil{\cdot}$ is the ceiling function: $\ceil x = \min\{y \in
  \setZ \mid y \geq x\}$.
\item If $W$ is a finite-dimensional linear space then $\End(W)$ is
  the set of linear endomorphisms of $W$ and $GL(W)$ denotes the group
  of invertible elements of $\End(W)$.
\item If $v \in W$ is a vector then $\Span v$ is the linear span of
  $v$ in $W$.  If $U \subset W$ is a nonempty subset then $\Span U$ is
  defined similarly.
\item If $A$ and $B$ are subspaces of $W$ then $A + B = \Span{A \cup
    B}$.
\end{enumerate}

\subsection{Colored vector spaces} \label{ssec:cvs}

Let $V$ be a finite-dimensional vector space over a field $\mathbb F$
with (not necessarily nonzero) vector subspaces $V_1, \dots, V_\n
\subset V$ such that $V = V_1 \oplus \dots \oplus V_\n$.  The tuple
$(V, V_1, \dots, V_\n)$ is an \n-\emph{colored vector space}.
Throughout this paper the symbol $V$ will refer to the vector space
$V$, together with the prescribed colored structure.  We will refer to
the elements of $\sub\n$ as \emph{colors}.  If $W \subset V$ is a
subset, we may write $W_i = W \cap V_i$.

\begin{definition}
  If $W \subset V$ is an arbitrary subset, the \emph{signature} of $W$
  is the function $\xi(W):\sub \n \to \setZ$ defined by $\xi_i(W) =
  \dim ( {\Span W}_i)$.  Observe that $\xi(W) = \xi(\Span W)$.
\end{definition}
      
\begin{lemma}
  If $W \subset V$ is a subspace then $\ds \dim W \geq \sum_{i = 1}^\n
  \xi_i(W)$.  If $U \subset W$ then $\xi_i(U) \leq \xi_i(W)$ for each
  $i$.
\end{lemma}

\begin{definition}
  We say that a subspace $W \subset V$ is \emph{colored} if $\dim W =
  \sum_{i=1}^\n \xi_i(W)$.  A vector $v \in V$ is colored if $\Span v$
  is colored.  A finite subset of $V$ is colored if each of its
  elements is colored.
\end{definition}

We can think of colored subspaces as those that lie ``squarely'' in
$V$, relative to $V_1, \dots, V_\n$.  For example, if $V = \setR^2$
with $V_1$ and $V_2$ the two coordinate axes then $(V, V_1, V_2)$ is a
colored vector space.  In this case, the only colored subspaces of $V$
are $0, V_1, V_2$, and $V$.  On the other hand, if $\n = 1$ and $V =
\setR^2$ then we have the colored vector space $(V, V)$ and each
subspace of $V$ is colored.

\begin{lemma} \
  \begin{enumerate}
  \item $V$ is colored with $\xi_i(V) = \dim V_i$.
  \item $0 \subset V$ is colored with $\xi_i(0) = 0$.
  \item If $W \subset V$ is a subspace then $W_1 + \dots + W_\n$ is
    the largest colored subspace of $W$ and $\xi(W) = \xi(W_1 + \dots
    + W_\n)$.

  \end{enumerate}
\end{lemma}

\begin{proposition} \label{lem:csub} If $W$ is a subspace of $V$ then
  the following are equivalent.
  \begin{enumerate}
  \item $W$ is colored,
  \item $W = W_1 + \dots + W_\n$,
  \item $(W, W_1, \dots, W_\n)$ is a colored vector space,
  \item $W$ has a colored basis,
  \item Each $w \in W$ can be written (uniquely) as $w = w_1 + \dots +
    w_\n$, with $w_i \in W_i$.
  \item If $w \in W$ is written $w = w_1 + \dots + w_\n$ with $w_i \in
    V_i$ then $w_i \in W$.
  \end{enumerate}
\end{proposition}

\begin{corollary}
  If $W \subset V$ is a subspace then there is a colored subspace $U
  \subset V$ such that $V = U \oplus W$.
\end{corollary}

\begin{proof}
  Let $U$ be any colored subspace such that $U + W = V$.  We know that
  such $U$ exist because $V$ is an example.  The proposition
  guarantees a colored basis $\catB$ for $U$.  We may also choose any
  basis $\catA$ of $W$.  If $U \cap W \neq 0$ then there is a
  nontrivial dependence relation among the elements of $\catA \cup
  \catB$.  Since $\catA$ is a linearly independent set, this
  dependence relation must nontrivially include an element $v \in
  \catB$.  Clearly, $U' = \Span{\catB \setminus \{v\}}$ is colored
  with $U + W = V$ and $\dim U' < \dim U$.  The result follows by
  induction.
\end{proof}

\begin{corollary}
  The set of colored vectors in $V$ is precisely $\ds \widetilde V =
  \bigcup_{i=1}^\n V_i$.
\end{corollary}

\begin{definition} \label{def:colorfunction} We define the ``color''
  function $\ds \chi: \widetilde V \setminus \{0\} \to \sub \n$ by
  $\chi(v) = i$, where $v \in V_i$.
\end{definition}

We mention here some standard results that we will use immediately.

\begin{lemma} \label{lem:sums} \
  \begin{enumerate}
  \item If $A, B \subset V$ are subspaces then $\dim A \cap B + \dim
    (A + B) = \dim A + \dim B$.
  \item Assume that $\ds \{a_i\}_{i=1}^\infty$ and $\ds
    \{b_i\}_{i=1}^\infty$ are sequences of real numbers satisfying
    $a_i \leq b_i$ for each $i \in \setN$.  If the series $\sum_{i
      =1}^\infty a_i$ and $\sum_{i=1}^\infty b_i$ are each convergent
    and their sums are equal then $a_i = b_i$ for each $i$.
  \end{enumerate}
\end{lemma}

\begin{lemma}
  If $A, B \subset V$ are colored subspaces then $A \cap B$ and $A +
  B$ are colored and $\xi(A + B) + \xi(A \cap B) = \xi(A) + \xi(B)$.
  If $A \cap B = 0$ then $\xi(A \oplus B) = \xi(A) + \xi(B)$.
\end{lemma}

\begin{proof}
  Obviously, $(A \cap B)_i = A_i \cap B_i$ and $A_i + B_i \subset (A +
  B)_i$, so
  \begin{align*}
    \dim A + \dim B &= \sum_{i=1}^\n \xi_i(A) + \sum_{i=1}^\n \xi_i(B) \\
    &= \sum_{i=1}^\n \left(\xi_i(A) + \xi_i(B)\right) \\
    &= \sum_{i=1}^\n \left( \dim (A_i) + \dim (B_i) \right)\\
    &= \sum_{i=1}^\n \left( \dim (A_i \cap B_i) + \dim (A_i + B_i) \right)\\
    &\leq \sum_{i=1}^\n \xi_i(A \cap B) + \sum_{i=1}^\n \xi_i(A + B) \\
    &\leq \dim (A \cap B) + \sum_{i=1}^\n \xi_i(A + B) \\
    &\leq \dim (A \cap B) + \dim (A + B).  \\
    &= \dim A + \dim B,
  \end{align*}
  so by lemma \ref{lem:sums} each inequality above is an equality and
  all of the claims follow.
\end{proof}

\begin{lemma}
  If $A \subset V$ is a colored subspace then $(V / A, V_1 / A_1,
  \dots, V_\n / A_\n)$ is a colored vector space, with $\xi(V / A) =
  \xi(V) - \xi(A)$.  If $W$ is a subspace of $V$ containing $A$ then
  $W$ is colored if and only if $W / A$ is colored.
\end{lemma}

Strictly speaking, in the above lemma $V_i / A_i$ should be
interpreted as $(V_i + A) / A$, but the isomorphism is clear.

\begin{lemma}
  A subset $\catB \subset V$ is a colored basis of $V$ if and only if
  $\catB_i$ is a basis of $V_i$ for each $i$.
\end{lemma}

\subsection{Colored change of basis}

$K = GL(V_1) \times \dots \times GL(V_\n) \subset GL(V)$ acts on $V$,
preserving $V_i$.  The orbits are parametrized by the power set of
$\sub \n$, so there are $2^\n$ orbits.  If $v \in V$ is written as $v
= v_1 + \dots + v_\n$, with $v_i \in V_i$, then the corresponding set
is $\{i \mid v_i \neq 0\}$.

More generally, $K$ acts on the set of subspaces of $V$.  In fact, if
$k \in K$ then $\chi(k \cdot v) = \chi(v)$ for all colored $v$.
Therefore, $\xi(W) = \xi(k \cdot W)$, so $W$ is colored if and only if
$k \cdot W$ is colored.  We conclude that this action restricts to a
signature-preserving action on colored subspaces.  We wish to
parametrize the orbits of this action---a task that will be easier
once we have established a definition, motivated by $\xi$.

\begin{definition}
  A \emph{signature} is a function $f: \sub \n \to \nneg$.  We define
  the \emph{size} of $f$ by $|f| = \sum_{i = 1}^\n f(i)$.  If $f$ and
  $g$ are signatures then we say that $f \leq g$ if $f(i) \leq g(i)$
  for each $i$.
\end{definition}

\begin{lemma} \
  \begin{enumerate}
  \item The set of signatures is a monoid partially ordered by $\leq$.
  \item If $f$, $g$, and $h$ are signatures then $f \leq g$ if and
    only if $f + h \leq g + h$.
  \item If $f \leq g$ are signatures then $|f| \leq |g|$.
  \item If $f$ and $g$ are signatures with $f \leq g$ then $|f| = |g|$
    if and only if $f = g$.
  \end{enumerate}
\end{lemma}

\begin{lemma}
  If $W \subset V$ is a subspace then
  \begin{enumerate}
  \item $\xi(W)$ is a signature.
  \item $W$ is colored if and only if $|\xi(W)| = \dim W$.
  \item If $U \subset W$ then $\xi(U) \leq \xi(W)$.
  \item If $U \subset W$ are subspaces satisfying $\xi(U) = \xi(W)$
    and $W$ is colored then $U = W$.
  \item If $f \leq \xi(W)$ is a signature then there is a colored
    subspace $U \subset W$ such that $\xi(U) = f$.
  \end{enumerate}
\end{lemma}

\begin{proposition}
  The set of orbits of the $K$-action on the set of subspaces of $V$
  is parametrized by signatures $f \leq \xi(V)$.  That is, if $U$ and
  $W$ are colored then they are $K$-conjugate if and only if $\xi(U) =
  \xi(W)$.  In particular, the set of orbits is finite.
\end{proposition}

This statement can be generalized further.  If $0 = f_0 < f_1 < \dots
< f_r = \xi(V)$ is a chain of signatures then we can apply the above
lemma to build a chain of colored subspaces $0 = W_0 \subset \dots
\subset W_r = V$ with $\xi(W_k) = f_k$.  $K$ naturally acts on such
colored partial flags and we might ask what the orbits are.  This is
straightforward, summarized in the following proposition, which is an
immediate consequence of proposition \ref{lem:K-basis}.

\begin{proposition}
  The set of $K$-orbits on partial flags of colored subspaces is
  finite and is parametrized by chains $0 = f_0 < f_1 < \dots < f_r =
  \xi(V)$ of signatures.  That is, two colored partial flags $0 = W_0
  \subsetneq \dots \subsetneq W_{r_1} = V$ and $0 = U_0 \subsetneq
  \dots \subsetneq U_{r_2} = V$ are $K$-conjugate if and only if $r_1
  = r_2$ and $\xi(W_k) = \xi(U_k)$ for each $k$.
\end{proposition}

\begin{proposition} \label{lem:K-basis} If $\catB = \{v_{i,j} \in V_i
  \mid 1 \leq i \leq n, 1 \leq j \leq \dim V_i\}$ and $\catB' =
  \{v_{i,j}' \in V_i \mid 1 \leq i \leq n, 1 \leq j \leq \dim V_i\}$
  are colored bases of $V$ then the automorphism of $V$ defined by
  $v_{i,j} \mapsto v_{i,j}'$ is in $K$.
\end{proposition}

\subsection{Colored endomorphisms}

\begin{definition}
  $x \in \End(V)$ is \emph{colored} if $xv$ is colored for every
  colored $v \in V$.
\end{definition}

\begin{proposition}
  If $x \in \End(V)$ then the following are equivalent:
  \begin{enumerate}
  \item $x$ is colored,
  \item $xW$ is colored for every colored subspace $W$,
  \item There is a function $\sigma: \sub \n \to \sub\n$ such that
    $xV_i \subset V_{\sigma(i)}$.
  \end{enumerate}
\end{proposition}

\begin{proof} \

  \mbox{(3) $\implies$ (1)} is obvious.  In fact, $\chi(xv) =
  \sigma(\chi(v))$ if both vectors are nonzero.

  \mbox{(2) $\implies$ (1)} is also immediate, for if $v$ is colored
  then $\Span v$ is colored, hence $x \Span v = \Span{xv}$ is colored.

  \mbox{(1) $\implies$ (2)} follows once we have chosen a colored
  basis for $W$.

  \mbox{(1) $\implies$ (3)} is proved by contrapositive.  Fix $i \in
  \sub \n$.  If $xV_i \neq 0$ then there are $v, w \in V_i$ such that
  $xv$ and $xw$ are nonzero and colored.  If $\chi(xv) \neq \chi(xw)$
  then $v + w$ is colored but $x(v + w) = xv + xw$ is not.  Therefore,
  $x$ is not colored.
\end{proof}

If $x V_i \subset V_{\sigma(i)}$ for each $i$ then we may say that $x
\in \End(V)$ is \emph{$\sigma$-colored}.  The set of all
$\sigma$-colored endomorphisms of $V$ is a linear space and contains
$\nil_\sigma$, the cone of nilpotent $\sigma$-colored endomorphisms of
$V$.  Note that the map $x \mapsto \sigma$ is well-defined only to the
extent that $xV_i \neq 0$.  That is, if $xV_i = 0$ then $\sigma(i)$
may be arbitrary.  Otherwise, $xV_i$ is well-defined.  This shows that
$\catN_\sigma \cap \catN_{\sigma'}$ is not empty.  In fact, the zero
transformation is in $\catN_\sigma$ for each $\sigma$.  If $\sigma$ is
the identity function and $x$ is $\sigma$-colored then we say that $x$
is \emph{trivially colored}.  Clearly, $K$ is precisely the set of
trivially colored automorphisms of $V$.

The equivalence of (1) and (3) brings us back to quiver
representations.  Since $\sigma: \sub \n \to \sub \n$, we can think of
$\sigma$ as a functional graph.  That is, the vertices are elements of
$\sub \n$ and the edges are precisely the pairs $(i, \sigma(i))$.  The
proposition shows that $x$ is $\sigma$-colored if and only if $x$ can
be thought of as a quiver representation of $\sigma$ with linear
spaces $V_i$ and maps $\ds x|_{V_i}: V_i \to V_{\sigma(i)}$.  While we
are really concerned with the case where $\sigma$ is an $\n$-cycle,
there are a few results that we can prove if $\sigma$ is not so
specialized.  With this perspective in mind, we can think of a colored
subspace $W$ as simply a choice of $(W_1, \dots, W_\n)$, with $W_i
\subset V_i$.

One nice property possessed by representations of functional graphs as
opposed to more general quivers is that there is a clear notion of
nilpotency that coincides with our usual understanding of nilpotency.
Since each vertex has exactly one outgoing edge we can choose bases
for $V_i$ and write the quiver representation as a matrix $A$.  The
representation is nilpotent if $A$ is nilpotent.
 
\begin{lemma}
  If $x \in \End (V)$ is $\sigma$-colored and $y \in \End(V)$ is
  $\tau$-colored then $x y$ is $\sigma \tau$-colored.  In particular,
  $x^k$ is $\sigma^k$-colored.
\end{lemma}

\begin{proof}
  $(x y) V_i = x (y V_i) \subset x V_{\tau(i)} \subset
  V_{\sigma(\tau(i))} = V_{\sigma \tau (i)}$.
\end{proof}

\begin{proposition} Assume that $W$ is colored and that $x \in \End V$
  is $\sigma$-colored, with $\sigma$ injective.  Then
  \begin{enumerate}
  \item $\ker x$ is colored.
  \item $x^{-1}(W)$ is colored.
  \item $\xi_{\sigma(m)}(x(W)) = \xi_m(W) - \xi_m(\ker x \cap W)$.
  \item $\xi_{\sigma(m)}(W) = \xi_m(x^{-1}(W)) - \xi_m(\ker x)$.
  \end{enumerate}
\end{proposition}
\begin{proof}
  To prove (2) let $v \in x^{-1}(W)$ and write $v = v_1 + \dots +
  v_\n$ with $v_i \in V_i$.  Then $xv = xv_1 + \dots + xv_\n$ is a
  decomposition with $xv_i \in V_{\sigma(i)}$.  Since $W$ is colored
  and $xv \in W$ we conclude that $x v_i \in W$, hence $v_i \in
  x^{-1}(W)$.  To prove (1) simply apply (2) to $W = 0$.
  
  Formula (3) is a simple application of the rank-nullity theorem to
  $x|_{W_m}$.  Injectivity of $\sigma$ is required to ensure that
  $W_{\sigma(m)} \cap x(W) = x(W_m)$.  Formula (4) is just (3) applied
  to $x^{-1} (W)$.
\end{proof}

\begin{corollary}
  If $x$ is colored and invertible then $x$ is $\sigma$-colored for
  some bijective $\sigma$ and $x^{-1}$ is $\sigma^{-1}$-colored.
\end{corollary}

\begin{proof}
  Let $\dim V_{i_0}$ be maximal.  Since $x$ is invertible,
  $\sigma(i_0)$ is well-defined and $\dim x(V_{i_0}) = \dim V_{i_0}$.
  But $x (V_{i_0}) \subset V_{\sigma(i_0)}$ and $\dim V_{i_0}$ is
  maximal, so $\dim V_{\sigma(i_0)} = \dim V_{i_0}$.  Inductively, if
  $V_i \neq 0$ then $\dim V_{\sigma(i)} = \dim V_i$.  If $V_i = 0$
  then we may choose $\sigma(i) = i$.  Invertibility of $x$ guarantees
  that $\sigma$ is invertible and the rest follows from (2).
\end{proof}

\begin{lemma}
  If $x$ is $\sigma$-colored and $A \subset V$ is an $x$-stable
  colored subspace then $x|_A$ is $\sigma$-colored relative to $(A,
  A_1, \dots, A_\n)$.  The quotient endomorphism $\overline x: V / A
  \to V / A$ is well-defined and is $\sigma$-colored relative to $(V /
  A, V_1 / A_1, \dots, V_\n / A_\n)$.
\end{lemma}

From this point on, we will assume that $\sigma(\n) = 1$ and
$\sigma(i) = i + 1$ for $i \neq\n$, so $\sigma$ is the cyclic graph of
order $\n$.  With this assumption, we suppress the dependence on
$\sigma$ and write $\nil = \nil_\sigma$.  When we say that an
endomorphism is colored, we will just assume that it is
$\sigma$-colored.  We call $\nil$ the \emph{colored nilpotent cone} of
$V$.  The natural action of $K$ on $V$ induces a change-of-basis
(conjugation) action on $\nil$.  We wish to classify the set $K
\backslash \nil$ of $K$-orbits on $\nil$.  That is, if $\orb_x = K
\cdot x$ is the orbit that contains $x$ and $y \in \nil$ is arbitrary,
we seek simple criteria for determining if $y \in \orb_x$.

For ease of notation, we think of the set $\sub\n$ of colors as the
group $\zmod$, so $\sigma(i) = i + [1]$.  As was mentioned in the
subsection on notation, we will choose $0$ as the preferred
representative of $[\n]$.

\section{The colored nilpotent cone} \label{sec:nil}

In this section we introduce the concept of a colored Jordan basis for
a colored nilpotent endomorphism of $V$.  This immediately leads to
the notion of a colored partition.  We show that the colored Jordan
basis gives a bijection between $K \backslash \nil$ and an appropriate
set of colored partitions.

\subsection{Colored Jordan bases}
\begin{definition}
  If $x \in \End(V)$ and $W \subset V$ is any nonempty subset then we
  say that $W$ is $x$-\emph{stable} if $x (W) \subset \left(W \cup
    \{0\}\right)$.
\end{definition}

Note that if $W$ is a subspace (or any other set containing $0$) then
$W$ is $x$-stable if and only if $x(W) \subset W$.

\begin{definition}
  If $x \in \End(V)$ is nilpotent then a \emph{Jordan basis for $x$}
  is an $x$-stable basis of $V$ that contains a basis of $\ker x$.
\end{definition}

\begin{definition}
  A \emph{partition} is a function $\lambda: \setN \to \nneg$ such
  that $\lambda_i \geq \lambda_{i+1}$ for each $i$ and $\lambda_i = 0$
  for some $i$.  We define the \emph{size} of $\lambda$ by $|\lambda|
  = \sum_{i = 1}^\infty \lambda_i$, a sum that is clearly finite, and
  the \emph{length} of $\lambda$ by $l(\lambda) = \# \{i \in \setN
  \mid \lambda_i > 0\}$.
\end{definition}

\begin{lemma} \label{lem:jordan} A basis $\catB$ of $V$ is a Jordan
  basis for a nilpotent $x \in
  \End(V)$ if and only if there is a (necessarily unique) partition
  $\lambda$ with $|\lambda| = \dim V$ such that the elements of
  $\catB$ can be labeled $v_{i,j}$ with the following properties:
  \begin{enumerate}
  \item $1 \leq i \leq l(\lambda)$,
  \item $1 \leq j \leq \lambda_i$,
  \item If $j > 1$ then $xv_{i, j} = v_{i, j - 1}$,
  \item $x v_{i, 1} = 0$,
  \item $l(\lambda) = \dim \ker x$.
  \end{enumerate}
\end{lemma}

\begin{proof}
  Assume that $\catB$ is a Jordan basis for $x$.  Since $x$ is
  nilpotent, there is some $v \in \catB$ with $xv = 0$.  By
  cardinality it cannot be the case that $x: \catB \to \catB \cup
  \{0\}$ is surjective, hence $\catB \setminus x\catB$ is nonempty.
  Let $v_1, \dots, v_r$ be the elements of $\catB \setminus x\catB$.
  Set $\lambda_i = \min \{k \mid x^kv_i = 0\}$.  By reordering, we may
  assume that $\lambda_i \geq \lambda_{i+1}$.  Set $v_{i, \lambda_i} =
  v_i$ and $v_{i, j} = x^{\lambda_i - j} v_{i, \lambda_i}$.
  Uniqueness of $\lambda$ and the reverse implication should be clear,
  for if $\lambda^t$ is the transpose partition then $\sum_{i=1}^k
  \lambda_i^t = \dim \ker x^k$.
\end{proof}

These properties of Jordan bases, as well as several that follow, are
classical; the important fact is that we can treat Jordan bases in the
usual way, even when we make the additional assumption that the basis
is colored.  Colored Jordan bases will be central to many of the
constructions we present throughout this paper.

\begin{lemma}
  Let $x \in \nil$ and assume that $A, B$ are $x$-stable colored
  subspaces of $V$ with $A \cap B = 0$.  If $\mathcal A, \catB$ are
  colored Jordan bases for $x|_A$ and $x|_B$, respectively, then
  $\mathcal A \cup \catB$ is a colored Jordan basis for $x|_{A \oplus
    B}$.
\end{lemma}

\begin{lemma}
  Let $\catB$ be a colored Jordan basis for $x \in \nil$ and let
  $\catA \subset \catB$ be $x$-stable.  If $A = \mathop{Span} \catA$
  then
  \begin{enumerate}
  \item $A$ is $x$-stable and colored;
  \item $\catA$ is a colored Jordan basis for $x|_A$.
  \item $\catB \setminus \catA$ is a colored Jordan basis for $x|_{V /
      A}$.  That is, $\{a + A \mid a \in \catB \setminus \catA\}$ is a
    colored Jordan basis for $x|_{V / A}$.
  \end{enumerate}
\end{lemma}

\subsection{Colored partitions}

In the same way that a Jordan basis naturally leads to a partition, a
colored Jordan basis naturally leads to a colored partition.  Suppose
that $x \in \nil$ has a Jordan basis $\catB = \{v_{i,j}\}$, labeled as
in lemma \ref{lem:jordanlabel}, that is colored.  From definition
\ref{def:colorfunction} we have the color function $\chi$, whose
codomain we now think of as $\zmod$.  If $0 < j < \lambda_i$ then
$\chi(v_{i,j}) = \chi(x v_{i,j+1}) = \chi(v_{i,j+1}) + [1]$.
Inductively, then, $\chi(v_{i,j}) = \chi(v_{i, \lambda_i}) +
[\lambda_i - j]$.  This equation shows that $\chi(v_{i,j})$ is
completely determined by the pair $(\lambda, \epsilon)$, where
$\epsilon_i = \chi(v_{i, \lambda_i})$ whenever $1 \leq i \leq
l(\lambda)$.  Note that if $\lambda_i = \lambda_j$ and $\epsilon_i
\neq \epsilon_j$ then we can interchange the roles of $i$ and $j$,
obtaining a new labeling of the same basis.  This leads to the
following definition.

\begin{definition}
  A $k$-\emph{colored partition} is a pair $(\lambda, \epsilon)$,
  where $\lambda$ is a partition and $\epsilon: \setN \to \zmod[k]$ is
  a function such that for each $m \in \zmod[k]$ there are infinitely
  many $i$ with $\epsilon_i = m$.  If $i \in \setN$ then the pair
  $(\lambda_i, \epsilon_i)$ is the $i$th \emph{row} of $(\lambda,
  \epsilon)$ and this row has \emph{length} $\lambda_i$ and
  \emph{color} $\epsilon_i$.  Two $k$-colored partitions are
  \emph{equivalent} if one can be obtained from the other by permuting
  rows of the same length.  The \emph{size} and \emph{length} of
  $(\lambda, \epsilon)$ are inherited from $\lambda$.
\end{definition}

The requirement that there are infinitely many $i$ with $\epsilon_i =
m$ is a technical convention whose main consequence is to make certain
constructions notationally easier.  It also ensures that there are
only finitely many equivalence classes of colored partitions of a
given size.  It also means that in most settings we can disregard the
value of $\epsilon_i$ if $\lambda_i = 0$, thinking of $(\lambda,
\epsilon)$ as a pair of finite tuples.  As $n$ is distinguished
throughout this paper, we may refer to an $\n$-colored partition as
simply a ``colored partition.''

We visualize a colored partition by drawing the (left-justified) Young
diagram for $\lambda$ and labeling the rightmost box in row $i$ with
$\epsilon_i$.  Labels then increase by $1 \pmod \n$ from right to left
across rows, so the color of the box in row $i$ (counting from the
top) and column $j$ (counting from the left) is given by $\epsilon_i +
[\lambda_i - j]$.  It is clear that the construction works in reverse:
each diagram constructed in this way comes from a unique colored
partition.  Two of these \emph{colored Young diagrams} are equivalent
if one can be obtained from the other by reordering rows of the same
length.

\begin{definition}
  The \emph{signature} of a colored partition $(\lambda, \epsilon)$ is
  the function $\xi(\lambda, \epsilon): \zmod \to \nneg$ defined by
  $\xi_m(\lambda, \epsilon) = \#\{(i, j) \mid 1 \leq j \leq \lambda_i,
  \epsilon_i + [\lambda_i - j] = m \}$.  For a fixed signature $f$ let
  $\catP_f$ denote the (finite) set of equivalence classes of colored
  partitions of signature $f$.  When writing it down, we may think of
  $\xi(\lambda, \epsilon)$ as the tuple $(\xi_0 (\lambda, \epsilon),
  \dots, \xi_{\n - 1}(\lambda, \epsilon))$.
\end{definition}

\begin{figure}[h]
  \caption{A $3$-colored partition of signature $(6, 7, 5)$.  In this
    example we have $\lambda = (5, 4, 4, 2, 2, 1)$ and $\epsilon = (0,
    0, 2, 1, 0, 1)$.  }
  \begin{center}
    \includegraphics{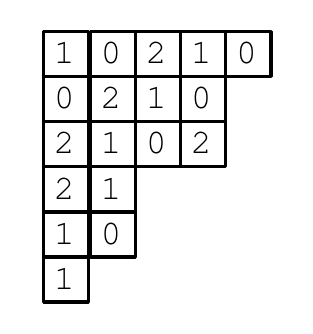}
  \end{center}
\end{figure}

\begin{definition}
  Let $(\lambda, \epsilon)$ be a colored partition with $\xi(\lambda,
  \epsilon) = \xi(V)$.  We say that a basis $\catB = \{v_{i,j}\}$ of
  $V$ is of type $(\lambda, \epsilon)$ if for each $v_{i,j} \in \catB$
  we have
  \begin{enumerate}
  \item $1 \leq i \leq l(\lambda)$,
  \item $1 \leq j \leq \lambda_i$,
  \item $\chi(v_{i,j}) = \epsilon_i + [\lambda_i - j]$.
  \end{enumerate}
\end{definition}

Strictly speaking, it is the \emph{labeled} set $\catB$ that is of
type $(\lambda, \epsilon)$.  However, the terminology given has the
advantage of brevity.  The colored Young diagram is a convenient way
to visualize $\catB$.  The coordinate $(i,j)$ gives a color-preserving
bijection between the boxes of the diagram and the elements of
$\catB$, so we may think of the boxes as elements of $\catB$.  If
$\catB$ happens to be a Jordan basis of $x \in \nil$ then we can
visualize the action of $x$ as sending each box to the one immediately
to its left.  Boxes in the leftmost column are sent to zero.

\begin{lemma} \label{lem:partmap} If $(\lambda, \epsilon)$ is a
  colored partition of signature $\xi(V)$ then
  \begin{enumerate}
  \item If $\catB$ is any colored basis of $V$ then the elements of
    $\catB$ can be labeled $v_{i,j}$ to make $\catB = \{v_{i,j}\}$ a
    basis of type $(\lambda, \epsilon)$;
  \item If $\catB = \{v_{i,j}\}$ is a basis of type $(\lambda,
    \epsilon)$ then $\catB$ is a colored Jordan basis for $x$, where
    $x$ is the colored endomorphism of $V$ defined by \[x v_{i,j} =
    \begin{cases} 0 & j = 1, \\ v_{i, j-1} & j > 1;
    \end{cases}\]
  \item If $x$ is defined as in $(2)$ and we define $\orb_{\lambda,
      \epsilon} = \orb_x$ then $(\lambda, \epsilon) \mapsto
    \orb_{\lambda, \epsilon}$ is a well-defined map from
    $\catP_{\xi(V)}$ into $\knil$.
  \end{enumerate}
\end{lemma}

We will see in the next section that the map $(\lambda, \epsilon)
\mapsto \orb_{\lambda, \epsilon}: \catP_{\xi(V)} \to \knil$ is a
bijection.  For now, we observe that the boxes in the leftmost column
of the colored Young diagram form a basis of $\ker x$.  Similarly, the
boxes in the first $k$ columns form a basis of $\ker x^k$.  We define
$s_k(x) = \xi(\ker x^k)$, the signature of the first $k$ columns of
the colored Young diagram corresponding to $x$.  The signatures $s_k$
are important combinatorial data that will be seen to completely
characterize orbits.

More generally, let $\lambda: \setN \to \nneg$ be any function with
finite support and let $\epsilon$ be as in the definition above.  We
can similarly visualize $(\lambda, \epsilon)$, though the rows may not
be in descending order and there may be gaps to indicate $i$ with
$\lambda_i = 0$.  The group of permutations of $\setN$ acts on the set
of such pairs $(\lambda, \epsilon)$ by $\sigma \cdot (\lambda,
\epsilon) = (\lambda \circ \sigma^{-1}, \epsilon \circ \sigma^{-1})$.
Each orbit of this action contains a colored partition and if
$\lambda$ and $\lambda \circ \sigma$ are both partitions then $\lambda
\circ \sigma = \lambda$.  In other words, if two colored partitions
are in the same orbit then one can be transformed into the other by
reordering rows of the same length.  Therefore, each orbit contains a
unique equivalence class of colored partitions.

While there is no need to introduce this level of generality here,
certain constructions later are simpler in this context.  They will
begin with a colored partition and produce an object that may not be a
colored partition but is equivalent to a colored partition.  The
description above gives us a well-defined (up to equivalence) way of
building a colored partition from such an object.

\begin{lemma} \label{lem:sig} If $(\lambda, \epsilon)$ is a colored
  partition and $m \in \zmod$ then
  \[\xi_m(\lambda, \epsilon) = \sum_{i=1}^\infty \ceil{\frac
    {\lambda_i - \rep{m - \epsilon_i}} \n},\] a formula that is
  invariant under the action of each permutation $\sigma$ of $\setN$.
\end{lemma}

\subsection{The colored Jordan normal form}

\begin{definition}
  Fix $(v, x) \in V \times \End (V)$.  We write $\setF[x](v)$ to
  denote the smallest $x$-stable subspace of $V$ containing $v$.  If
  $x^k v = 0$ for some $k \in \nneg$, let $d_x(v)$ be the smallest
  such $k$.
\end{definition}

\begin{lemma} \label{lem:jordanlabel} Let $(v,x) \in V$ satisfy $x^k v
  = 0$ for some $k \in \setN$.  If $\catB_{v, x} = \{x^k v \mid 0 \leq
  k < d_x(v) \}$ then
  \begin{enumerate}
  \item $\catB_{v, x}$ is a Jordan basis for $x |_{\setF[x](v)}$, so
    $d_x(v) = \dim \setF[x](v) \leq \dim V$.
  \item $\setF[x](w) \subset \setF[x](v)$ if and only if $w \in
    \setF[x](v)$.
  \item The $x$-stable subspaces of $\setF[x](v)$ are precisely $x^i
    \setF[x](v) = \setF[x](x^iv) = \ker (x^{d_x(v) - i}) \cap
    \setF[x](v), 0 \leq i \leq d_x(v)$, with $\dim \setF[x](x^iv) =
    d_x(v) - i$.
  \item If $v$ is colored then $\catB_{v, x}$ is colored, so
    $\setF[x](v)$ is colored.
  \item If $\setF[x](v)$ is colored then there is a colored vector $w$
    such that $\setF[x](w) = \setF[x](v)$.  If $w'$ is another such
    vector then $\chi(w) = \chi(w')$.
  \item If $w \in \setF[x](v)$ then there exists $v' \in V$ such that
    $\setF[x](v') = \setF[x](v)$ and $w \in \catB_{v', x}$.  If
    $\setF[x](v)$ is colored and $w$ is colored then we may choose
    $v'$ to be colored.
  \end{enumerate}
\end{lemma}

\begin{proof} Since $v$ and $x$ are fixed, we will set $d = d_x(v)$
  throughout the proof to simplify notation.
  \begin{enumerate}
  \item It is clear that $\catB_{v, x}$ must be contained in any
    $x$-stable subspace of $V$ containing $v$.  The set $\catB_{v, x}$
    is $x$-stable because $x^dv = 0$, so its span must be
    $\setF[x](v)$.  We prove linear independence by induction on $d$.
    If $\sum_{j = 0}^{d - 1} a_j x^j v = 0$ then \[ 0 = x \sum_{j =
      0}^{d - 1} a_j x^j v = \sum_{j = 0}^{d - 2} a_j x^j (xv).\] By
    induction, we must have $a_j = 0$ for each $j < d - 1$.
    Therefore, $a_{d - 1}x^{d-1} v = 0$, hence $a_{d - 1} = 0$.  The
    rest follows immediately.
  \item This is obvious.
  \item From (1) the given spaces are $x$-stable and $\dim
    \setF[x](x^iv) = d - i$.  Let $w \in \setF[x](v)$ with $d' =
    d_x(w)$.  Write $w = \sum_{j=1}^d a_jx^{d-j}v$.  Then \[ 0 =
    x^{d'}w = \sum_{j=1}^d a_jx^{d+d'-j}v = \sum_{j=d' + 1}^d
    a_jx^{d+d'-j}v.  \]

    Linear independence implies that $a_{d' + 1}= \dots = a_{d} = 0$,
    so \[w = \sum_{j=1}^{d'} a_jx^{d-j}v = \sum_{j=1}^{d'}
    a_jx^{d'-j}(x^{d - d'}v).\] Therefore, $w \in
    \setF[x](x^{d-d'}v)$.  Since $\dim \setF[x](x^{d-d'}v) = d_x(w)$
    we must have $\setF[x](x^{d-d'}v) = \setF[x](w)$.
  \item $x^i v \in V_{\chi(v) + [i]}$.
  \item Write $v = v_1 + \dots + v_n$ with $v_m \in V_m$.  There must
    be some $r$ such that $d_x(v_r) \geq d_x(v)$.  Since $\setF[x](v)$
    is colored, we may set $w = v_r \in \setF[x](v)$.  Then $d_x(v_r)
    = \dim \setF[x](v_r) \leq \dim \setF[x](v) = d_x(v)$, hence
    $d_x(v_r) = d_x(v)$ and we apply (3).  Uniqueness of $r$ follows
    immediately from the fact that $\ker x|_{\setF[x](v)}$ is a
    one-dimensional colored subspace.
  \item If $w = \sum_{j=1}^{d'} a_jx^{d'-j}(x^{d - d'}v)$ as in (3),
    set $v' = \sum_{j=1}^{d'} a_jx^{d'-j}v$.  To prove the last claim
    we first observe that, since $\setF[x](v)$ is colored, we may
    assume that $v$ is colored.  In the above expression for $w$, the
    indices $j$ such that $a_j \neq 0$ must all be congruent modulo
    $\n$.  This congruence must also hold in the expression for $v'$,
    so $v'$ is colored.  \qedhere
  \end{enumerate}
\end{proof}

\begin{proposition} \label{prop:colbasis} Each element of $\nil$
  admits a colored Jordan basis.
\end{proposition}

\begin{proof}
  Fix $x \in \nil$.  Choose $v \in V$ such that $d_x(v)$ is maximal
  and decompose $v = v_1 + \dots + v_n$.  There is a $k$ such that
  $d_x(v_k) = d_x(v)$.  By relabeling, then, we may assume that $v$ is
  colored, so $W = \setF[x](v)$ is colored and $x$-stable with a
  colored Jordan basis $\catB_{v, x}$.
    
  Inductively assume that $W \subset V$ is an $x$-stable colored
  subspace that admits a colored Jordan basis.  That is, there exist
  colored vectors $v_1, \dots, v_r$ such that $\bigsqcup \catB_{v_i,
    x}$ is a basis of $W$.  Assume further that if $ x: V / W \to V /
  W$ is the map induced by $x$ then $d_x(v_i) \geq d_{\bar x} (\bar
  w)$ for each $\bar w \in V / W$.  Note that clearly $d_x(w) \geq
  d_{\bar x}(\bar w)$.
    
  Let $w$ be a colored vector with $d_{\bar x} (\bar w)$ maximal.
  Then $W \cap \setF[x](w)$ is an $x$-stable colored subspace of
  $\setF[x](w)$, hence $W \cap \setF[x](w) = \setF[x](x^kw)$ for some
  $k \geq 0$.  Write $x^kw = \sum u_i$, with $u_i \in \setF[x](v_i)$.
  By applying (6) from lemma \ref{lem:jordanlabel} to $u_i \in
  \setF[x](v_i)$ we may write $x^kw = \sum_i x^{k_i}v_i$.  Now,
  $d_x(v_i) - k_i = d_x(x^{k_i} v_i) \leq d_x(x^k w) = d_x(w) - k$ and
  we have $k_i - k \geq d_x(v_i) - d_x(w) \geq 0$.  Therefore, we can
  set $v_{r+1} = w - \sum_i x^{k_i - k}v_i$.  Then $W +
  \setF[x](v_{r+1}) = W + \setF[x](w)$ and $d_x(v_{r+1}) = d_{\bar
    x}(\bar w)$.  Furthermore, the construction ensures that $v_{r+1}$
  is colored, so $\left(\bigsqcup \catB_{v_i, x}\right) \sqcup
  \catB_{v_{r+1}, x}$ is a colored Jordan basis for $W \oplus
  \setF[x](w)$ and $d_x(v_i) \geq d_{\bar x} (\bar w)$ for each $\bar
  w \in V / (W \oplus \setF[x](w))$, which completes the induction.
\end{proof}

If $x$ has a colored Jordan basis of type $(\lambda, \epsilon)$ then
we may refer to (the equivalence class of) $(\lambda, \epsilon)$ as
the colored Jordan type of $x$.  We will shortly see that this is
well-defined.  With this terminology in mind, the proposition and its
proof give us the following:

\begin{corollary}
  If $v_0 \in V$ is colored and satisfies $d_x(v_0) \geq d_x(v)$ for
  each $v \in V$ then $W_0 = \setF[x](v_0)$ has an $x$-stable colored
  complement $W$ and $x |_W$ has the same colored Jordan type as $x
  |_{V / W_0}$.
\end{corollary}

\begin{theorem} \label{thm:paramc} The map $(\lambda, \epsilon)
  \mapsto \orb_{\lambda, \epsilon}: \catP_{\xi(V)} \to K \backslash
  \nil$ defined in lemma \ref{lem:partmap} is a bijection.  That is,
  if $x \in \nil$ has a colored Jordan basis of type $(\lambda,
  \epsilon)$ and $y \in \orb_x$ has a colored Jordan basis of type
  $(\alpha, \beta)$ then any $(\lambda, \epsilon)$ and $(\alpha,
  \beta)$ are equivalent.  Each colored partition of signature
  $\xi(V)$ is the type of a colored Jordan basis for some $x \in
  \nil$.  Moreover, if $x, y \in \nil$ then $\orb_y = \orb_x$ if and
  only if $s_k(x) = s_k(y)$ for each $k \in \setN$.
\end{theorem}

\begin{proof}
  Surjectivity is the content of proposition \ref{prop:colbasis}, so
  we only need to show injectivity.  If $h \in K$ and $k \in \setN$
  then
  \begin{align*}
    s_k(h \cdot x) &= \xi(\ker (h \cdot x)^k) \\
    &= \xi(\ker (h \cdot x^k) \\
    &= \xi(h \cdot \ker x^k) \\
    &= \xi(\ker x^k) \\
    &= s_k(x).
  \end{align*}
  Therefore, if $y \in \orb_x$ then $s_k(x) = s_k(y)$ for each $k \in
  \setN$.

  Now, if $s_k(x) = s_k(y)$ and we draw the colored Young diagram with
  the columns aligned on the left then the number of boxes in column
  $k$ is equal to $\dim (\ker x^k) - \dim \ker x^{k-1}$.  For a fixed
  color $m$, the number of boxes of color $m$ in column $k$ is
  precisely $s_k(x)(m) - s_{k-1}(x)(m)$.  If there is a box of color
  $m$ in column $k > 1$ then the box immediately to the left must be
  of color $m + [1]$, so inductively the rows are uniquely determined,
  up to reordering entire rows.  Therefore, the colored Jordan types
  of $x$ and $y$ are equivalent, so the map is injective.

  Finally, if the colored Jordan types of $x$ and $y$ are equivalent
  then $\orb_x = \orb_y$ by (3) in lemma \ref{lem:partmap}.
\end{proof}

\begin{definition} \label{def:s} If $x \in \nil$ has colored Jordan
  type $(\lambda, \epsilon)$ then $s_k(\lambda, \epsilon) = s_k(x)$.
\end{definition}

\begin{corollary}
  The following are equivalent:
  \begin{enumerate}
  \item $(\lambda, \epsilon)$ and $(\alpha, \beta)$ are equivalent;
  \item $\orb_{\lambda, \epsilon} = \orb_{\alpha, \beta}$;
  \item $s_k(\lambda, \epsilon) = s_k(\alpha, \beta)$.
  \end{enumerate}
\end{corollary}

If $\n = 1$ then $\epsilon$ is trivial, so we naturally obtain the
classical parametrization of nilpotents by partitions.  In this case,
the signature of a partition is the same as its size.  If $\n = 2$
then it is customary to use $+$ and $-$ as colors, rather than $0$ and
$1$, respectively, hence the terminology ``signed partition.''  The
signature of a signed partition is the pair $(p, q)$, where $p$ is the
number of boxes containing a $+$ sign and $q$ is the number of boxes
containing a $-$ sign.

\begin{figure}[h]
  \caption{A signed partition of signature $(8, 10)$.  This example is
    $\lambda = (5, 5, 3, 2, 2, 1)$ and $\epsilon = (-, +, -, -, -,
    -)$.}
  \begin{center}
    \includegraphics{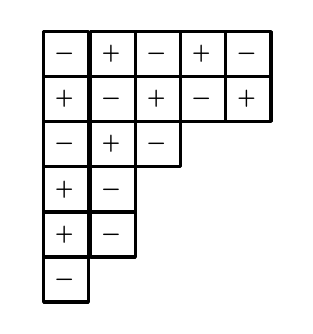}
  \end{center}
\end{figure}

\section {The enhanced colored nilpotent cone} \label{sec:enhnil}

Since the action of $K$ on $V$ preserves $\widetilde V = \bigcup_{i
  \in \zmod} V_i$, we have a diagonal action of $K$ on the
\emph{enhanced} colored nilpotent cone $\vnil$.  We have seen that
$\knil$ is finite and is parametrized by $\catP_{\xi(V)}$.  We will
show that this enhanced diagonal action also yields finitely many
orbits and we will describe a simple generalization of
$\catP_{\xi(V)}$ that parametrizes these orbits.  As was discussed
earlier, the case $\n = 1$ was proved in \cite{travkin} and \cite{AH},
with orbits parametrized by bipartitions.  The procedures and notation
used in \cite{AH} prove to generalize particularly well in this
context, so whenever possible we use them as a model in this
exposition.

\subsection{Marked colored partitions} \label{ssec:marked}

\begin{definition} \label{def:mark} If $(\lambda, \epsilon)$ is a
  colored partition and $k$ is a positive integer then
  \begin{enumerate}
  \item A \emph{marking} of $\lambda$ is a function $\mu: \setN \to
    \setZ$ such that $\mu_i \leq \lambda_i$ for each $i$.  The pair
    $(\lambda, \mu)$ is a \emph{marked partition}.  For convenience we
    will frequently make use of $\nu = \nu(\lambda, \mu) = \lambda -
    \mu \geq 0$.
  \item The triple $(\lambda, \epsilon, \mu)$ is a \emph{marked}
    colored partition.
  \item If $(\lambda, \mu)$ is a marked partition such that $0 \leq
    \mu_{i + 1} \leq \mu_i$ and $\nu_{i + 1} \leq \nu_i$ for each $i$
    then $(\lambda, \mu)$ is a \emph{bipartition} and $(\lambda,
    \epsilon, \mu)$ is a \emph{colored bipartition}.
  \item If $-k < \mu_i$ for all $i$ and $\mu_j < \mu_i + k $ and
    $\nu_j < \nu_i + k$ for each $i < j$ then $(\lambda, \mu)$ is a
    \emph{\qb[k]{}}.
  \item $(\lambda, \epsilon, \mu)$ is a \emph{\cqb[k]{}} if $(\lambda,
    \mu)$ is a \qb[k] and $\epsilon_i + \nu_i \equiv \epsilon_j +
    \nu_j \pmod k$ for each $i, j$.
  \item If $(\lambda, \epsilon, \mu)$ satisfies $-k < \mu_i$ for all
    $i$ then $(\lambda, \epsilon, \mu)$ is a \emph{\gqb[k]{}} if
    $\mu_j < \mu_i + k $ and $\nu_j < \nu_i + k$ for each pair $i < j$
    such that $\epsilon_i + \nu_i \equiv \epsilon_j + \nu_j \pmod k$.
  \end{enumerate}
\end{definition}

If $\n = 1$ then $\epsilon$ is trivial, so when it is convenient we
may simply express $(\lambda, \epsilon, \mu)$ as the marked partition
$(\lambda, \mu)$.  We visualize $(\lambda, \epsilon, \mu)$ by drawing
the colored Young diagram for $(\lambda, \epsilon)$ and marking the
wall between boxes $\mu_i, \mu_i + 1$ in row $i$.  If $\mu_i \leq 0$
then we mark the leftmost wall in the row.  As in the previous
section, we think of the boxes in the diagram as elements of a colored
Jordan basis $\catB = \{v_{i,j} \mid 1 \leq j \leq \lambda_i\}$ for
some $x \in \orb_{\lambda, \epsilon}$.  We may refer to $(\lambda,
\epsilon)$ as the \emph{shape} of $(\lambda, \epsilon, \mu)$.

It is usually convenient to draw the marked colored Young diagram with
the marks aligned.  That is, the rows are shifted left or right as
necessary so that the marks form a single vertical line.  If $\mu_i <
0$ then we draw the left end of the row $|\mu_i|$ positions to the
\emph{right} of the mark.  The marking divides the colored Young
diagram into two sub-diagrams.  If $\mu \geq 0$ then the sub-diagram
on the left of the marks is the colored diagram corresponding to
$(\mu, \epsilon + [\nu])$, while the diagram on the right corresponds
to $(\nu, \epsilon)$.  Finally, $(\lambda, \epsilon, \mu)$ is a
colored bipartition if and only if $(\mu, \epsilon + [\nu])$ and
$(\nu, \epsilon)$ are each colored partitions.

\begin{figure}[h]
  \caption{The \cqb[3] defined by $\lambda = (5,5,3,3,2,1)$, $\epsilon
    = (2,1,1,0,0,0)$, $\mu = (1,3,1,0,-1,1)$.  Notice that, on each
    side of the diagram, each increase in length (from top to bottom)
    is less than $3$.}
  \begin{center}
    \includegraphics{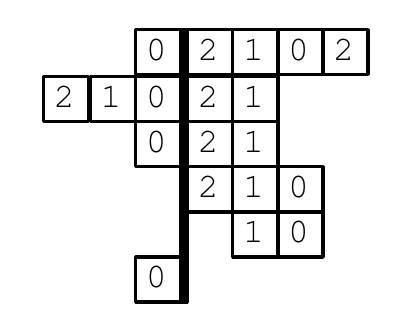}
  \end{center}
\end{figure}

Note that every \cqb[k] is automatically a \gqb[k] and every
bipartition is a \qb[k].  More generally, if $l > k$ then each \qb[k]
is automatically an \qb[l].  Furthermore, if $k = 1$ then $\mu_i \geq
0$ and $\mu_{i + 1} < \mu_i + k \implies \mu_{j+1} < \mu_i + 1$, so
$\mu_{i+1} \leq \mu_i$.  Similarly, $\nu_{i+1} < \nu_i + k \implies
\nu_{i+1} \leq \nu_i$.  That is, a \qb[1] is just a bipartition.  In
this sense, a \qb[k] is a deformation of a bipartition.  The following
lemma makes this idea precise.

\begin{lemma} \label{lem:rho2} If $(\lambda, \mu)$ is a marked
  partition, define the marking $\tilde \mu$ of $\lambda$ by \[\tilde
  \mu_i = \max(\{\mu_j \mid j \geq i\} \cup \{ \lambda_i - \lambda_j +
  \mu_j \mid j < i\} \cup \{0\}).\]
  \begin{enumerate}
  \item $(\lambda, \tilde \mu)$ is a bipartition satisfying $\tilde
    \mu \geq \mu$.
  \item If $(\lambda, \delta)$ is another bipartition satisfying
    $\delta \geq \mu$ then $\tilde \mu \leq \delta$.
  \item $(\lambda, \mu)$ is a \cqb[k] if and only if $0 \leq \tilde
    \mu - \mu < k$.
  \end{enumerate}
\end{lemma}

\begin{proof} \
  \begin{enumerate}
  \item It is obvious that $\tilde \mu \geq \mu$.  Now, for fixed $i$
    we have
    \begin{align*}
      \tilde \mu_i &= \max (\{\mu_j \mid j \geq i + 1\} \cup
      \{\lambda_i - \lambda_k + \mu_k  \mid  k  <  i  \}  \\
      & \qquad \cup  \{  \mu_i  \} \cup \{0\})  \\
      \tilde \mu_{i+1} &= \max (\{\mu_j \mid j \geq i + 1\} \cup
      \{\lambda_{i + 1} - \lambda_k + \mu_k \mid k < i \} \\
      & \qquad \cup \{\mu_i - \lambda_i + \lambda_{i+1} \} \cup
      \{0\}).
    \end{align*}
    
    These decompositions make it clear that for each element in the
    set corresponding to $\tilde \mu_{i+1}$ there is an element of the
    set corresponding to $\tilde \mu_i$ that is at least as large.
    Therefore, $\tilde \mu_i \geq \tilde \mu_{i+1}$.  A similar
    decomposition shows that $\lambda_i - \tilde \mu_i \geq
    \lambda_{i+1} - \tilde \mu_{i+1}$:
    \begin{align*}
      \lambda_i - \tilde \mu_i &= \min (\{\lambda_i - \mu_j \mid j
      \geq i + 1\}
      \cup \{\lambda_k -\mu_k \mid k < i \} \\
      & \qquad \cup \{ \lambda_i - \mu_i \} \cup \{\lambda_i\}) \\
      \lambda_{i+1} - \tilde \mu_{i+1} &= \min (\{\lambda_{i+1} -
      \mu_j \mid j \geq i
      + 1\} \cup \{ \lambda_k - \mu_k \mid k < i \} \\
      & \qquad \cup \{ \lambda_i - \mu_i \} \cup \{\lambda_{i+1}\}).
    \end{align*}

  \item Let $(\lambda, \delta)$ be a bipartition such that $\delta
    \geq \mu$.  If $j \geq i$ then $\delta_i \geq \delta_j \geq
    \mu_j$.  Similarly, if $j < i$ then $\delta_i = \lambda_i -
    (\lambda_i - \delta_i) \geq \lambda_i - (\lambda_j - \delta_j)
    \geq \lambda_i - \lambda_j + \mu_j = \lambda_i - \lambda_j +
    \mu_j$.  Therefore, $\delta \geq \max(\{\mu_j \mid j \geq i\} \cup
    \{ \lambda_i - \lambda_j + \mu_j \mid j < i\}) = \tilde \mu_i$

  \item Assume first that $0 \leq \tilde \mu - \mu < k$.  Let $i < j$.
    Then $\mu_j < \tilde \mu_j \leq \tilde \mu_i < \mu_i + k$ and
    $\lambda_j - \mu_j = \lambda_j - \mu_j < \lambda_j + k - \tilde
    \mu_j \leq k + \lambda_i - \tilde \mu_i \leq k + \lambda_i - \mu_i
    = \lambda_i - \mu_i + k$, so $(\lambda, \mu)$ is an \qb.
    Conversely, assume that $(\lambda, \mu)$ is an \qb.  If $j < i$
    then $\lambda_i - \lambda_j + \mu_j = \lambda_i + k - \lambda_i +
    \mu_i < k + \mu_i$.  If $j > i$ then $\mu_j < \mu_i + k$.
    Therefore, $\tilde \mu_i < \mu_i + k$ and we conclude that $0 \leq
    \tilde \mu - \mu < k$.  \qedhere
  \end{enumerate}
\end{proof}

\begin{definition}
  If $(\lambda, \mu)$ is a marked partition and \[\tilde \mu_i =
  \max(\{\mu_j \mid j \geq i\} \cup \{ \lambda_i - \lambda_j + \mu_j
  \mid j < i\} \cup \{0\})\] then $(\lambda, \tilde \mu)$ is the
  \emph{minimal bipartition} associated to $(\lambda, \mu)$.
\end{definition}

As usual, we view two marked colored partitions as equivalent if one
can be transformed into the other by reordering rows, along with
corresponding marks.  It is a simple exercise to show that if
$(\lambda, \epsilon, \mu)$ and $(\alpha, \beta, \gamma)$ are row
equivalent and one of them is an \cqb then so is the other.  Let $\ds
\catP_f^{\text m}$ denote the set of equivalence classes of marked
colored partitions of signature $f$.  Let $\ds \catQ_f \subset
\catP_f^{\text m}$ denote the (clearly finite) subset consisting of
\cqb{}s.  It will soon be important to consider a slightly stronger
equivalence relation on marked colored partitions, so when clarity is
required we may say ``row-equivalence'' to refer to the above
relation.

It is worth digressing here for a brief discussion of notation.  It is
common to define a bipartition as a pair $(\mu; \nu)$ of partitions
and then define $\lambda = \mu + \nu$.  This is done, for example, in
\cite{AH}.  To be consistent with this choice of notation, we could
define a colored bipartition to be a pair $((\mu, \beta); (\nu,
\epsilon))$ of colored partitions such that $\beta = \epsilon +
[\nu]$.  Alternatively, we could choose to denote this $(\mu, \nu,
\epsilon)$.  However, we find the notation in the definition, which
emphasizes the underlying partition $\lambda$, to be more convenient
for our purposes here.

Our parametrization of $\kvnil$ will essentially be in terms of a set
of marked colored partitions.  In fact, to each element of
$\catP_{\xi(V)}$ there corresponds an orbit in $K \backslash (V \times
\nil)$.  The set of marked colored partitions of signature $\xi(V)$ is
infinite, but we will see that the set of orbits corresponding to
marked colored partitions is finite, so it is clear from the outset
that there are many markings of a fixed colored partition that must be
considered equivalent for the purposes of this parametrization.  The
construction we give will make it clear that if $\mu_i \leq 0$ then
the precise value of $\mu_i$ is irrelevant.  Thus, we can consider
$(\lambda, \epsilon, \mu)$ and $(\alpha, \beta, \gamma)$ equivalent if
there is a permutation $\sigma$ of $\setN$ such that $\alpha = \lambda
\circ \sigma$, $\beta = \epsilon \circ \sigma$, and $\gamma_i = (\mu
\circ \sigma)_i$ whenever $\gamma_i > 0$ or $(\mu \circ \sigma)_i >
0$.  In other words, we are completely disregarding the value of
$\mu_i$ if $\mu_i \leq 0$.  Let $\widetilde \catP_f$ denote the set of
classes under this equivalence and let $\widetilde \catQ_f$ be the
subset whose classes each contain at least one \cqb.

If a signature $f$ is fixed then $\widetilde \catP_f$ and $\widetilde
\catQ_f$ are finite.  This is because from each class in $\widetilde
\catP_f$ we can always select an element $(\lambda, \epsilon, \mu)$
with $\mu \geq 0$.  In fact, this element is unique up to row
equivalence.  However, certain calculations are easier if we select a
different representative.  We will never actually use representatives
with $\mu_i \leq -\n$ in this exposition, but the fact that each class
is rich with representatives keeps notation simple and ensures a
framework for easily stating and proving the theorems in this section.
We observe here that each class in $\widetilde \catP_f$ is a union of
classes in $\ds \catP_f^{\text m}$.

We now explore the extent to which two \cqb{}s $(\lambda, \epsilon,
\mu)$ and $(\alpha, \beta, \gamma)$ can lie in different classes in
$\widetilde \catQ_f$.  By reordering we may assume $\alpha = \lambda$,
$\beta = \epsilon$, and that if $\mu_i \neq \gamma_i$ then $\mu_i \leq
0$ and $\gamma_i \leq 0$.  If $\mu_{i_0} > 0$ for some $i_0$ and
$\mu_i \leq 0$ then $\epsilon_i + [\lambda_i - \mu_i] = \epsilon_{i_0}
+ [\lambda_{i_0} - \mu_{i_0}]$, so $ [\mu_i] = \epsilon_i -
\epsilon_{i_0} + [\lambda_i - \lambda_{i_0} + \mu_{i_0}]$ and $-\n <
\mu_i \leq 0$.  But this uniquely determines $\mu_i$.  Therefore, if
$\mu_i > 0$ for some $i$ then there is only one equivalence class of
\cqb{}s in each element of $\widetilde \catQ_f$.  If, however, $\mu_i
\leq 0$ for each $i$ then the same calculation shows that $\mu$ is
fixed once we have chosen a value of $\mu_1$.  Thus, there are exactly
$\n$ (row equivalence classes of) \cqb{}s $(\lambda, \epsilon, \mu)$
satisfying $\mu \leq 0$, determined by $m = \epsilon_1 + [\lambda_1 -
\mu_1]$.

\begin{definition}  
  Let $(\lambda, \epsilon, \mu)$ be a marked colored partition and fix
  $m \in \zmod$.  For each $i$ let $\delta_i = \max \{k \in \setZ \mid
  k \leq \mu_i, \epsilon_i + [\lambda_i - k] = m\}$ and let $(\lambda,
  \tilde \mu)$ be the minimal bipartition corresponding to $(\lambda,
  \mu)$.  Define $\ds \rho_m: \catP_f^{\text m} \to \catP_f^{\text m}$
  by $\rho_m(\lambda, \epsilon, \mu) = (\lambda, \epsilon, \delta)$
  and $\ds \bar \rho:\catP_f^{\text m} \to \catP_f^{\text m}$ by $\bar
  \rho(\lambda, \epsilon, \mu) = (\lambda, \epsilon, \tilde \mu)$.
\end{definition}

It is clear that $\rho_m$ and $\bar \rho$ are simply processes that
produce a new marking of a given colored partition.  In terms of our
diagrams (with marks aligned), $\rho_m$ modifies the picture by
shifting each row to the right just until each column consists of a
single color and the column immediately to the left of the marks has
color $m$.  On the other hand, $\bar \rho$ shifts rows to the left
just far enough to produce a bipartition.  Note that $\rho_m \circ
\rho_m = \rho_m$ and that $(\rho_m \circ \bar \rho) (\lambda,
\epsilon, \mu) = (\lambda, \epsilon, \mu)$ if and only if $(\lambda,
\epsilon, \mu)$ is a \cqb and $\epsilon + [\lambda - \mu] = m$.  With
this notation, lemma \ref{lem:rho2} can be restated as follows:

\begin{lemma} \label{lem:rho} A marked colored partition $(\lambda,
  \epsilon, \mu)$ is a \cqb if and only if there is a color $m$
  satisfying $\epsilon + [\lambda - \mu] = m$ and $(\lambda, \epsilon,
  \mu) = \rho_m(\lambda, \epsilon, \delta)$ for some colored
  bipartition $(\lambda, \epsilon, \delta)$.  Moreover, among such
  markings $\delta$ of $\lambda$ there is a unique \emph{minimal}
  marking $\tilde \mu$ of $\lambda$ such that $\bar \rho(\lambda,
  \epsilon, \mu) = (\lambda, \epsilon, \tilde \mu)$ is a colored
  bipartition satisfying $\tilde \mu \leq \delta$ for each $i$.
\end{lemma}

In other words, $\rho_m$ and $\bar \rho$ are inverse bijections
between the set of \cqb{}s (and their corresponding equivalence
classes) and the corresponding set of minimal colored bipartitions.
We will employ either of these sets as convenience dictates.

\subsection{Normal bases}

We now show how to construct an enhanced $K$-orbit from a marked
colored partition.

\begin{definition} \label{def:psi} Let $(\lambda, \epsilon, \mu)$ be a
  marked colored partition and let $x \in \orb_{\lambda, \epsilon}$.
  Let $\catB = \{v_{i,j}\}$ be a colored Jordan basis for $x$ of type
  $(\lambda, \epsilon)$.  Extend this notation, setting $v_{i,j} = 0$
  if $j \leq 0$.  Define $\Psi(\lambda, \epsilon, \mu) = \orb_{v, x}$,
  where $v = \sum_{i = 1}^{l(\lambda)} v_{i, \mu_i}$.
\end{definition}

It should be clear that $\ds \Psi: \widetilde \catP^{\text m}_{\xi(V)}
\to V \times \nil$ is well-defined.  As was mentioned above,
$\widetilde \catP^{\text m}_{\xi(V)}$ is finite.  Corollary
\ref{cor:dim} implies that if $\n = 2$ and $\dim V_0 = \dim V_1 = p$
then $\dim \nil = 2p^2 - p$, so $\dim V \times \nil = 2p^2 + p$.
Since $\dim K = 2 p^2$ there is no hope that $K \backslash (V \times
\nil)$ is finite, so $\Psi$ is clearly not surjective.  This is the
case in general if $\n > 1$.  We will, however, see that $\kvnil$ is
always contained in the image of $\Psi$.

Our goal now is to determine when two marked colored partitions are in
the same fiber of $\Psi$.  As might be guessed from the terminology
introduced earlier in this section, the answer is related to \cqb{}s.
We will see that if $\orb \in \kvnil$ then the fiber of $\Psi$ over
$\orb$ consists of a single class in $\widetilde Q_{\xi(V)}$.

\begin{definition} \label{def:normal} If $(v, x) \in V \times \nil$
  then a \emph{normal basis} for $(v, x)$ is a colored Jordan basis
  $\catB = \{v_{i,j}\}$ for $x$ such that if $(\lambda, \epsilon)$ is
  the type of $\catB$ then there is a marking $\mu$ of $\lambda$ such
  that
  \begin{enumerate}
  \item $\ds v = \sum_{i = 1}^{l(\lambda)} v_{i,\mu_i}$,
  \item $(\lambda, \epsilon, \mu)$ is a \gqb.
  \end{enumerate}
\end{definition}

In general, not every element of $V \times \nil$ admits a normal
basis.  In fact, if $(v, x)$ admits a normal basis with corresponding
\gqb $(\lambda, \epsilon, \mu)$ then $\Psi(\lambda, \epsilon, \mu) =
\orb_{v, x}$.  So, if $(v, x)$ admits a normal basis then $\orb_{v,
  x}$ is in the image of $\Psi$.  We will see that the converse is
true, as well: if $\orb_{v, x}$ is in the image of $\Psi$ then $(v,
x)$ admits a normal basis.  As a first step, we observe the following
lemma, which suggests that the existence of a normal basis is an
important orbit invariant.

\begin{lemma} \label{lem:normal} If $(v, x)$ admits a normal basis
  then so does each element of $\orb_{v, x}$.  Conversely, if $(v, x)$
  and $(w, y)$ each admit a normal basis corresponding to the same
  \gqb then $\orb_{v, x} = \orb_{w, y}$.
\end{lemma}

\begin{proof}
  Let $\catB$ be a normal basis for $(v, x)$ with corresponding \gqb
  $(\lambda, \epsilon, \mu)$.  If $k \cdot (v, x) = (w, y)$ then $k
  \cdot \catB$ is a normal basis for $(w, y)$ with corresponding \gqb
  $(\lambda, \epsilon, \mu)$.  Conversely, if we fix normal bases for
  $(v, x)$ and $(w, y)$ corresponding to the same \gqb then the
  obvious change of basis transformation lies in $K$.
\end{proof}

\begin{definition}
  Let $(\lambda, \epsilon, \mu)$ be a marked colored partition and let
  $\catB = \{v_{i,j}\}$ be a colored basis of type $(\lambda,
  \epsilon)$.  Then we write \[ \catB^\mu = \{v_{i,j} \in \catB \mid 1
  \leq j \leq \mu_i \}.\]
\end{definition}

\begin{lemma} \label{lem:sum} Let $x \in \nil$ have a colored Jordan
  basis $\catB = \{ v_{i,j} \}$ of type $(\lambda, \epsilon)$ and let
  $\mu$ be a marking of $\lambda$.  Then
  \begin{enumerate}
  \item $\Span{\catB^\mu}$ is colored and $x$-stable;
  \item $\catB^\mu$ is a colored Jordan basis for $x
    |_{\Span{\catB^\mu}}$ of type $(\mu, \epsilon + [\lambda - \mu])$;
  \item $\catB \setminus \catB^\mu$ is a colored Jordan basis for
    $x|_{V / {\Span{\catB^\mu}}}$ of type $(\lambda - \mu, \epsilon)$;
  \item If $\mu_i \in \{0, \lambda_i\}$ for each $i$ then
    $\Span{\catB^{\lambda - \mu}}$ is $x$-stable and $x
    |_{\Span{\catB^{\lambda - \mu}}}$ and $x |_{V / \Span{\catB^\mu}}$
    have the same colored Jordan type.
  \end{enumerate}
\end{lemma}

We may speak of \emph{deleting} a row or collection of rows from a
partition, colored partition, or marked colored partition.  Let
$\iota_k: \setN \to \setN$ be defined by \[\iota_k(i) = \begin{cases}
  i & i < k \\ i + 1 & i \geq k.
\end{cases}\] To delete row $k$ from $(\lambda, \epsilon, \mu)$ is to
construct $\Delta_k(\lambda, \epsilon, \mu) = (\lambda \circ \iota_k,
\epsilon \circ \iota_k, \mu \circ \iota_k)$.  The deletion of row $k$
from a partition or colored partition is performed analogously.  If $S
\subset \setN$ is finite, we may delete from $(\lambda, \epsilon,
\mu)$ all the rows indexed by elements of $S$ in the obvious way: Let
$a_1 < a_2 < \dots < a_r$ be the elements of $S$.  We simply construct
$\Delta_S(\lambda, \epsilon, \mu) = (\lambda \circ \iota, \epsilon
\circ \iota, \mu \circ \iota)$, where $\iota = \iota_{a_1} \circ
\cdots \circ \iota_{a_r}$.  The order of the composition is
significant here, because $\iota_k$ and $\iota_{k'}$ do not commute if
$k \neq k'$.  If $k' < k$ then $\iota_k \circ \iota_{k'} =
\iota_{k'+1} \circ \iota_k$.

\begin{lemma} \label{lem:sub} Let $\catB = \{v_{i,j}\}$ be a normal
  basis for $(v, x) \in V \times \nil$ with corresponding \gqb
  $(\lambda, \epsilon, \mu)$ and let $S \subset \sub{l(\lambda)}$.  If
  we set $\mu_i = \begin{cases} \lambda_i & i \in S \\ 0 & i \not \in
    S.  \end{cases}$ and $A = \Span{\catB^\mu}$ then $\catB^{\lambda -
    \mu}$ is a normal basis for $(v + A, \bar x) \in (V /A) \times
  \nil(V / A)$ with corresponding \gqb $\Delta_S (\lambda, \epsilon,
  \mu)$.
\end{lemma}

\begin{theorem} \label{thm:param} \
  \begin{enumerate}
  \item The image of $\Psi$ (definition \ref{def:psi}) is precisely
    the set of enhanced $K$-orbits whose elements admit a normal basis
    (definition \ref{def:normal}).  That is, each fiber of $\Psi$
    contains a \gqb (definition \ref{def:mark}).
  \item $\Psi: \widetilde \catQ_{\xi(V)} \to \kvnil$ is a bijection.
    That is,
    \begin{enumerate}
    \item If $\orb \in K \backslash (V \times \nil)$ then $\orb \in
      \kvnil$ if and only if there is a \cqb $(\lambda, \epsilon,
      \mu)$ such that $\Psi(\lambda, \epsilon, \mu) = \orb$;
    \item If $(v, x) \in \vnil$ and $v \neq 0$ then any two \cqb{}s
      that correspond to $\orb_{v, x}$ are identical, up to permuting
      rows.
    \item If $x \in \nil$ and $v = 0$ then the \cqb{}s corresponding
      to $\orb_{v, x}$ are precisely $\rho_m(\lambda, \epsilon, 0)$,
      where $(\lambda, \epsilon)$ is the colored Jordan type of $x$.
    \end{enumerate}
  \end{enumerate}
\end{theorem}

\begin{proof}
  We use the proof in \cite{AH} as a model.  In fact, the only
  obstacle to following this proof exactly is that we must be careful
  to preserve the colored structure of $V$.  The procedure described
  below gives a simple algorithm for producing the \gqb associated to
  $(v, x) \in \vnil$.

  To prove (1) we observe, first of all, that if $\orb_{v, x} =
  \Psi(\lambda, \epsilon, \mu)$ then we can trivially assume that
  $\mu_i > -\n$ for each $i$.  Let $\catB = \{v_{i,j}\}$ be a colored
  Jordan basis for $x$ of type $(\lambda, \epsilon)$ such that $v =
  \sum v_{i, \mu_i}$.  We will iteratively modify $\catB$ until $\mu_j
  < \mu_i + \n$ and $\nu_j < \nu_i + \n$ for each $i < j$ such that
  $\epsilon_i + [\nu_i] = \epsilon_j + [\nu_j]$.  Suppose there exists
  a pair $i < j$ that fails.  Note that, since $\mu_i + \nu_i =
  \lambda_i$ and $\lambda_i \geq \lambda_j$, we cannot have both
  $\mu_i + \n \leq \mu_j$ and $\nu_i + \n \leq \nu_j$.
        
  If $\mu_i + \n \leq \mu_j$ then for each $r$ define
  \begin{align*}
    \ds w_{k,r} &= \begin{cases}
      v_{i, \mu_i + \n} + v_{i, \mu_i} & k = i \\
      v_{j,r} - v_{i, r - \mu_j + \mu_i + n} & k = j \\
      v_{k, r} & k \neq i, j.
    \end{cases}
  \end{align*}
  Then $\{w_{i,j}\}$ is a colored Jordan basis for $x$ of type
  $(\lambda, \epsilon)$ and \[v = \sum_k v_{k, \mu_k} = \sum_{k \neq
    i} w_{k, \mu_k} + w_{i, \mu_i + \n}.\] Therefore, we have
  effectively redefined $\mu_i$ to be $\mu_i + \n$, leaving $\mu$
  otherwise unchanged.  Pictorially, we have moved the mark in row $i$
  to the right by $\n$ positions.

  If $\nu_i + \n \leq \nu_j$, define
  \begin{align*}
    \ds w_{k,r} &= \begin{cases}
      v_{i,r} - v_{j, r - \mu_i + \mu_j + n} & k = i \\
      v_{j, \mu_j + \n} + v_{j, \mu_j} & k = j \\
      v_{k, r} & k \neq i, j.
    \end{cases}
  \end{align*}
  By similar reasoning, this effectively redefines $\mu_j$ to be
  $\mu_j + \n$.  Pictorially, we have moved the mark in row $j$ to the
  right by $\n$ positions.

  We repeat this step as long as it is possible.  The condition
  $\epsilon_i + [\nu_i] = \epsilon_j + [\nu_j]$ ensures that this
  change of basis can be accomplished by an element of $K$.  The
  condition $\lambda_i \geq \lambda_j$ plus $\mu_i + \n \leq \mu_j$
  (resp.  $\nu_i + \n \leq \nu_j$) ensures that each iteration results
  in a valid marking of $\lambda$, i.e., $\mu_i \leq \lambda_i$ for
  each $i$.  Each iteration also increases the quantity $\sum_{i,
    \lambda_i > 0} \mu_i \leq |\lambda|$, so this process must
  eventually terminate, yielding the appropriate inequalities.  Note
  that each iteration also preserves the quantity $\epsilon + [\lambda
  - \mu]$.

  To prove (a) we fix $(v, x) \in \vnil$ and let $\catB = \{v_{i,j}\}$
  be a colored Jordan basis for $x$ of type $(\lambda, \epsilon)$.  If
  $v = 0$ then $\orb_{v, x} = \Psi(\rho_m(\lambda, \epsilon, 0))$ for
  each $m$.  Otherwise, set $m = \chi(v)$, $v = \sum_{i,j} a_{i,j}
  v_{i,j}$, and $v_i = \sum_j a_{i,j} v_{i,j}$.  By applying (6) from
  lemma \ref{lem:jordan} to each Jordan block, noting that $v_i$ is
  colored, we may assume that $v_{i, \lambda_i}$ is colored and $v_i =
  x^{\nu_i} v_{i, \lambda_i}$ for some $0 \leq \nu_i \leq \lambda_i$.
  If $v_i \neq 0$ then $\chi(v_i) = m$.  Otherwise, redefine $\nu_i =
  \min\{t \in \setZ \mid t \geq \lambda_i, \epsilon_i + [t] = m \}$.
  Then $\Psi(\lambda, \epsilon, \mu) = \orb_{v, x}$, where $\mu =
  \lambda - \nu$.  Note that by construction we have $\epsilon +
  [\lambda - \mu] = m$, so the algorithm in (1) yields a \cqb.

  We now wish to show that $\ds \Psi|_{\widetilde \catQ_{\xi(V)}}$ is
  injective.  Let $(v, x) \in \vnil$ and let $\catB = \{v_{i,j}\}$ be
  a normal basis for $(v, x)$ with \cqb $(\lambda, \epsilon, \mu)$.
  Since $v = \sum_i v_{i, \mu_i}$ it is clear that if $v = 0$ then
  $\mu_i \leq 0$.  But if a color $m$ is fixed then for each $i \in
  \setN$ there is a unique $\mu_i$ satisfying $- \n < \mu_i \leq 0$
  and $\epsilon_i + [\lambda_i] - [\mu_i] = m$, so $(\lambda,
  \epsilon, \mu) = \rho_m(\lambda, \epsilon, 0)$.  As $m$ varies,
  these \cqb{}s all lie in the same equivalence class in $\widetilde
  \catQ_{\xi(V)}$ and (c) is proved.

  We may, therefore, assume that $v \neq 0$.  Since $v = \sum_i v_{i,
    \mu_i}$, lemma \ref{lem:jordan} implies that $\dim \setF[x](v) =
  \max \{\mu_i \mid 1 \leq i \leq l(\lambda)\}$.  Therefore, there is
  an integer $i$ with $\mu_i = \dim \setF[x](v)$.  Since $(\lambda,
  \epsilon, \mu)$ is an \qb we have $[\dim \setF[x](v)] = [\mu_i] =
  \epsilon_i + [\lambda_i] - \chi(v)$.  We can, therefore, set $k =
  \min\{i \mid \epsilon_i + [\lambda_i] = [\dim \setF[x](v)] +
  \chi(v)\}$, noting that this expression is independent of $\mu$.  By
  congruence there is an integer $j$ such that $\mu_k = \mu_i + j \n$.
  But $k \leq i$, so $\mu_k + \n > \mu_i$, so $j \n > -\n$, i.e., $j >
  -1$, hence $j \geq 0$ and $\mu_k \geq \mu_i = \dim \setF[x](v)$.
  But maximality of $\mu_i$ forces $\mu_k \leq \mu_i = \dim
  \setF[x](v)$.  Therefore, $\mu_k = \dim \setF[x](v)$.  In other
  words, the marking of the longest row of $(\lambda, \epsilon)$
  satisfying $\epsilon_k + [\lambda_k] = [\dim \setF[x](v)] + \chi(v)$
  is forced upon us.
  
  Set $S= \{k\}$ and build $A$ as in lemma \ref{lem:sub}.  Then
  $\Delta_k(\lambda, \epsilon, \mu)$ is a \cqb that corresponds to
  $x|_{V / A}$.  Inductively, the \cqb corresponding to $x|_{V/A}$ is
  unique, so $\mu_i$ is also completely determined if $i \neq k$.
  There is one case that must be considered carefully.  If $v \in A$
  then $v + A \in V / A$ is the zero vector.  We saw above that there
  are $\n$ markings $\delta$ of $\Delta_k(\lambda, \epsilon)$ that are
  valid in this case. However, there is only one satisfying $\epsilon
  + [\lambda - \delta] = m$, proving (b).
\end{proof}

\begin{corollary}
  If $m \in \zmod$ is fixed then $K \backslash (V_m \times \nil)$ is
  in bijection with the set of \cqb{}s $(\lambda, \epsilon, \mu)$ of
  signature $\xi(V)$ such that $\epsilon + [\lambda - \mu] = m$, via
  the map $\Psi$.

\end{corollary}

\begin{corollary}
  Let $(\lambda, \epsilon)$ be a colored partition and let $m \in
  \zmod$.  Then
  \begin{enumerate}
  \item $\ds V_m \times \orb_{\lambda, \epsilon} =
    \bigsqcup_{\substack{\mu \\ \epsilon + [\lambda - \mu] = m \\
        (\lambda, \epsilon, \mu) \in \catQ_{\xi(V)}}} \orb_{\lambda,
      \epsilon, \mu}$
  \item $\ds \orb_{\rho_m(\lambda, \epsilon, 0)} = \{0\} \times
    \orb_{\lambda, \epsilon} \cong \orb_{\lambda, \epsilon}$.
  \end{enumerate}
\end{corollary}

\begin{figure}[h]
  \caption{All the orbits in $K \backslash (V_0 \times \nil)$, for $\n
    = 2$ and signature $(2, 2)$, parametrized by signed \qb and ranked
    by dimension.  The bottommost orbit is zero.  The next orbit up
    has dimension 2.  The topmost orbits each have dimension 8.  An
    edge indicates that the lower orbit lies in the Zariski closure of
    the upper orbit.}
  \begin{center}
    \includegraphics[scale=.55555555]{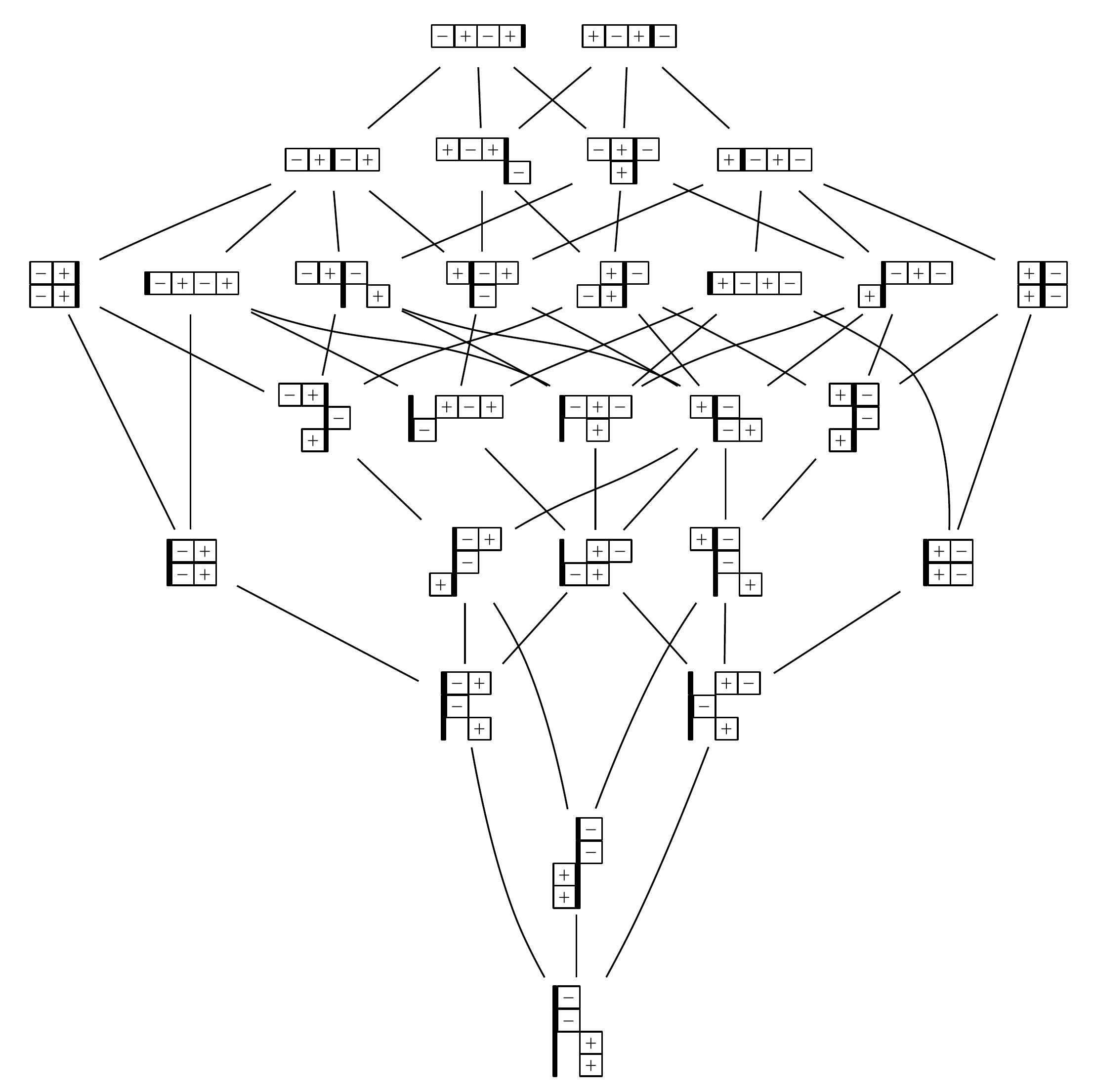}
  \end{center}
\end{figure}

\subsection{Connections to classical orbits}

\begin{proposition} \label{prop:rhobar} Let $m \in \zmod$ and let
  $(\lambda, \epsilon, \mu)$ be a marked colored partition satisfying
  $\epsilon + [\lambda - \mu] = m$.  Let $(\lambda, \epsilon, \bar
  \mu)$ be a \cqb that shares a fiber of $\Psi$ with $(\lambda,
  \epsilon, \mu)$.  Then $\bar \rho (\lambda, \epsilon, \mu) = \bar
  \rho(\lambda, \epsilon, \bar \mu)$, so $(\lambda, \epsilon, \bar
  \mu) = \rho_m(\bar \rho (\lambda, \epsilon, \mu))$.
\end{proposition}

\begin{proof}
  Let $\delta$ be a marking of $\lambda$ obtained from $\mu$ by one
  step of the iterative portion of the proof of theorem
  \ref{thm:param}.  Let $(\lambda, \epsilon, \tilde \mu) = \bar \rho
  (\lambda, \epsilon, \mu)$ and $(\lambda, \epsilon, \tilde \delta) =
  \bar \rho (\lambda, \epsilon, \delta)$.  We will show that $\tilde
  \mu = \tilde \delta$.  Therefore, for a fixed orbit the marking
  $\tilde \mu$ is the same, regardless of the representative marking
  used to construct $\tilde \mu$.
  
  If $(\lambda, \epsilon, \mu)$ is not a \cqb then there exist $s < r$
  with either $\mu_s + \n \leq \mu_r$ or $\nu_s + \n \leq \nu_r$.  We
  need to show that if $\delta$ is constructed in either of these
  cases then $\tilde \delta = \tilde \mu$.  The second case is
  entirely analogous to the first, so we will only prove the first
  case.  Assume that $s < r$ and $\mu_s + \n \leq r$.  Then \[\delta_k
  =
  \begin{cases} \mu_s + \n & k = s \\ \mu_k & k \neq s.
  \end{cases}\]
    
  The formulas for $\tilde \mu$ and $\tilde \delta$ make it clear that
  $\tilde \mu \leq \tilde \delta$.  On the other hand, the same
  formulas show that if $\tilde \delta_k > \tilde \mu_k$ then either
  $k \leq s$ and $\tilde \delta_k = \delta_s$ or $k > s$ and $\tilde
  \delta_k = \delta_s + \lambda_k - \lambda_s$.  We divide our effort
  into three cases:
  \begin{enumerate}
  \item If $k \leq s$ then $k < r$ and $\tilde \delta_k = \delta_s =
    \mu_s + \n \leq \mu_r \leq \tilde \mu_k$.
  \item If $s < k < r$ then $\tilde \delta_k = \delta_s + \lambda_k -
    \lambda_s \leq \delta_s = \mu_s + \n \leq \mu_r \leq \tilde
    \mu_k$.
  \item If $k > r$ then $\tilde \delta_k = \delta_s + \lambda_k -
    \lambda_s = \mu_s + \n + \lambda_k - \lambda_s \leq \mu_r +
    \lambda_k - \lambda_s \leq \mu_r + \lambda_k - \lambda_r \leq
    \tilde \mu_k$.
  \end{enumerate}
  In each case we have a contradiction, so $\tilde \delta_k \leq
  \tilde \mu_k$ for each $k$ and we have $\tilde \delta = \tilde \mu$.
  Inductively, we just need to apply an adequate number of iterations
  until we arrive at the \cqb.  The last claim follows because $\rho_m
  \circ \bar \rho$ fixes \cqb{}s.
\end{proof}
  
\begin{corollary}
  If $\epsilon + [\lambda - \mu] = \epsilon + [\lambda - \delta] = m$
  then $\orb_{\lambda, \epsilon, \mu} = \orb_{\lambda, \epsilon,
    \delta}$ if and only if $\bar \rho(\lambda, \epsilon, \mu) = \bar
  \rho(\lambda, \epsilon, \delta)$.  So, if $m$ is fixed then $K
  \backslash (V_m \times \nil)$ is parametrized by minimal
  bipartitions.
\end{corollary}

\begin{lemma}
  Let $k$ be a divisor of $\n$ and let $\zeta_k: \zmod \to \zmod[k]$
  be the natural projection.  For $m \in \zmod[k]$ define $\ds W_m =
  \bigoplus_{\zeta_k(i) = m}^{\n / k} V_i $.  Then
  \begin{enumerate}
  \item $(V, W_{[0]}, \dots, W_{[k - 1]})$ is a $k$-colored vector
    space.
  \item If $x \in \nil$ then $x W_{m} \subset W_{m + [1]}$ for each $m
    \in \zmod[k]$.
  \end{enumerate}
\end{lemma}

That is, $x$ is colored relative to the subspaces $W_i$.  Moreover,
$K$ naturally embeds in $GL(W_0) \times \dots \times GL(W_{n-1})$.  In
other words, if we reduce the number of colors to $k$ (combining all
colors that are congruent modulo $k$) then we get a new colored
nilpotent cone and we can view $x$ inside this larger cone.  Since the
identity map $V \to V$ is $K$-equivariant, we have an induced map
$\Phi_k$ of orbits.  On the other hand, we have an obvious map of
marked colored partitions that reduces the number of colors to $k$:
$\phi_k(\lambda, \epsilon, \mu) = (\lambda, \zeta_k \circ \epsilon,
\mu)$.  Unsurprisingly, these two maps are compatible.  The case $k =
1$ is especially illuminating.
  
\begin{proposition}
  If $(\lambda, \epsilon, \mu)$ is any marked partition then $\Psi
  \circ \phi_k = \Phi_k \circ \Psi $.  That is, $\Phi_k(\orb_{\lambda,
    \epsilon, \mu}) = \orb_{\phi_k(\lambda, \epsilon, \mu)}$.
\end{proposition}

\begin{corollary}
  If $(\lambda, \epsilon, \mu)$ is a \cqb and the minimal bipartition
  of $(\lambda, \mu)$ is $(\lambda, \tilde \mu)$ then $\phi_1
  (\lambda, \epsilon, \mu) = (\lambda, \tilde \mu)$, the bipartition
  given by Achar-Henderson.
\end{corollary}

On the other hand, we have a natural $K$-equivariant projection
$\theta: \vnil \to \nil$.  It should be clear that
$\theta(\orb_{\lambda, \epsilon, \mu}) = \orb_{\lambda, \epsilon}$.
In other words, our parametrization is well-behaved relative to each
setting that we are trying to generalize.  It projects in the most
natural way possible to the colored nilpotent cone and to the setting
explored in \cite{AH}.

\subsection{An alternative parametrization of enhanced orbits}

Fix an orbit $\orb$ in the image of $\Psi$ and let $(\lambda,
\epsilon)$ be the corresponding colored partition.  Then \[S_\orb = \{
\mu \geq 0 \text{ a marking of } \lambda \mid \Psi(\lambda, \epsilon,
\mu) = \orb\}\] is partially ordered by the rule $\delta \leq \mu$ if
$\delta_i \leq \mu_i$ for each $i$.  Since $S_\orb$ is finite and
nonempty, $S_\orb$ has at least one minimal element.  A primary
objective of this subsection is to show that the minimal element is
unique up to row equivalence.  Throughout this subsection, if $\mu \in
S_\orb$ then let $\tilde \mu \in S_\orb$ be defined by the usual
formula $\tilde \mu_i = \max \left(\{\mu_j \mid j \geq i\} \cup \{
  \lambda_i - (\lambda_j - \mu_j) \mid j \leq i\} \right)$.

\begin{lemma} \label{lem:char1} If $\mu \in S_\orb$ is minimal and $i
  < j$ satisfy $\mu_i > 0$, $\mu_j > 0$, and $\epsilon_j + [\lambda_j
  - \mu_j] = \epsilon_i + [\lambda_i - \mu_i]$ then $\mu_i > \mu_j$
  and $\lambda_i - \mu_i > \lambda_j - \mu_j$.  In particular,
  $\lambda_i \geq \lambda_j + 2$.
\end{lemma}

\begin{proof}
  Define
  \begin{align*}
    \delta_k &= \begin{cases} \max\{\mu_i - \n, 0\} & k = i \\ \mu_k &
      k \neq i,
    \end{cases} \\
    \gamma_k &= \begin{cases} \max\{\mu_j - \n, 0\} & k = j \\ \mu_k &
      k \neq j. \end{cases}
  \end{align*}
  
  If $\mu_i \leq \mu_j$ then $\delta < \mu$ and the algorithm in
  theorem \ref{thm:param} shows that $\delta \in S$.  On the other
  hand, if $\lambda_i - \mu_i \leq \lambda_j - \mu_j$ then $\gamma <
  \mu$ and $\gamma \in S$.  In either case, minimality of $\mu$ is
  violated.  Now, if $\mu_i > \mu_j$ and $\lambda_i - \mu_i >
  \lambda_j - \mu_j$ then $\mu_i \geq \mu_j + 1$ and $\lambda_i -
  \mu_i \geq \lambda_j - \mu_j + 1$.  We just add these two
  inequalities to prove the last claim.
\end{proof}

\begin{lemma} \label{lem:char2} If $\mu \in S_\orb$ is minimal and
  $\mu_i > 0$ then $\tilde \mu_i = \mu_i$.
\end{lemma}

\begin{proof}
  By lemma \ref{lem:char1}, if $j > i$ then $\mu_j < \mu_i$, so
  $\tilde \mu_i = \max \{ \lambda_i - (\lambda_j - \mu_j) \mid j \leq
  i\} = \lambda_i - \min \{ \lambda_j - \mu_j \mid j \leq i\}$.
  Again, the lemma shows that if $j < i$ then either $\mu_j = 0$, so
  $\lambda_j - \mu_j = \lambda_j \geq \lambda_i \geq \lambda_i -
  \mu_i$, or $\lambda_j - \mu_j > \lambda_i - \mu_i$.  Therefore,
  $\min \{ \lambda_j - \mu_j \mid j \leq i\} = \lambda_i - \mu_i$ and
  the claim is proved.
\end{proof}

\begin{theorem}
  Let $\orb$ be in the image of $\Psi$ and let $(\lambda, \epsilon)$
  be a corresponding colored partition.  Then
  \begin{enumerate}
  \item There is a \emph{minimal} marking $\mu$ of $\lambda$
    satisfying
    \begin{enumerate}
    \item $\Psi(\lambda, \epsilon, \mu) = \orb$;
    \item $\mu \geq 0$;
    \item If $\delta \leq \mu$ is any marking of $\lambda$ satisfying
      \emph {(a)} and \emph{(b)} then $\delta = \mu$.
    \end{enumerate}
  \item If $\mu$ satisfies \emph {(a)} and \emph{(b)} then there
    exists $\delta \leq \mu$ that is minimal in the sense of
    \emph{(c)}.
  \item If $\mu$ and $\delta$ are each minimal then $(\lambda,
    \epsilon, \mu)$ and $(\lambda, \epsilon, \delta)$ are
    row-equivalent.
  \item If $\mu$ satisfies \emph {(a)} and \emph{(b)} then $\mu$ is
    minimal if and only if $\mu_i > \mu_j$ and $\lambda_i - \mu_i >
    \lambda_j - \mu_j$ for every pair $i < j$ satisfying $\mu_i > 0$,
    $\mu_j > 0$, and $\epsilon_j + [\lambda_j - \mu_j] = \epsilon_i +
    [\lambda_i - \mu_i]$.
  \end{enumerate}
\end{theorem}

\begin{proof}
  Claim (1) is just a restatement of the fact that $S_\orb$ contains
  at least one minimal element.  Claim (2) follows from the proof of
  lemma \ref{lem:char1} once we have proved (4).  We will show that
  any $\mu$ and $\delta$ satisfying the inequalities given in (4) must
  be equivalent.  The rest follows immediately from lemma
  \ref{lem:char1} because any minimal marking must satisfy these
  inequalities.
  
  We begin with the case $\orb \in \kvnil$.  First, observe that
  $\Psi(\lambda, \epsilon, \mu) = \Psi(\lambda, \epsilon, \delta)$
  forces $\tilde \mu = \tilde \delta$.  So, if $\mu_i \neq \delta_i$
  then by lemma \ref{lem:char2} exactly one of these must be zero.
  Let $i$ be the smallest index with $\mu_i \neq \delta_i$.  We may
  assume with no loss of generality that $\mu_i > 0$ and $\delta_i =
  0$.
  
  Since $\tilde \delta_i = \tilde \mu_i = \mu_i > 0 = \delta_i$, there
  is either $k < i$ with $\lambda_i - (\lambda_k - \delta_k) = \mu_i$
  or $j > i$ with $\delta_j = \mu_i$.  In the first case, $\lambda_i -
  \mu_i = \lambda_k - \delta_k$.  By minimality of $i$ we have $\mu_k
  = \delta_k$, so $\lambda_i - \mu_i = \lambda_k - \mu_k$.  By lemma
  \ref{lem:char1} we must have $\mu_k = 0$, so $\lambda_k = \lambda_i
  - \mu_i < \lambda_i$, a contradiction.
  
  We conclude that there exists $j > i$ with $\delta_j = \mu_i >
  \mu_j$, so $\mu_j = 0$.  Now, if $k > j$ is arbitrary then $\mu_k <
  \mu_i = \tilde \mu_j$, so $\tilde \mu_j = \max\{ \lambda_j -
  (\lambda_k - \mu_k) \mid k < j \}$.  Therefore, there exists $k < j$
  with $\mu_i = \lambda_j - (\lambda_k - \mu_k) \leq \mu_k$, hence
  $\mu_k > 0$ and $k \leq i$.  Now, $\mu_i = \lambda_j - (\lambda_k -
  \mu_k) \leq \lambda_i - (\lambda_k - \mu_k)$, hence $\lambda_i -
  \mu_i \geq \lambda_k - \mu_k$.  Since $\mu_i > 0$ and $\mu_k > 0$,
  we must have $k \geq i$.

  Since $k = i$ we have $\mu_i = \lambda_j - (\lambda_i - \mu_i)$,
  hence $\lambda_j = \lambda_i$.  Now, $\epsilon_j + [\lambda_j -
  \mu_j] = \epsilon_i + [\lambda_i - \mu_i]$, so $\epsilon_j =
  \epsilon_i + [\mu_j - \mu_i] = \epsilon_i + [\delta_j - \mu_i] =
  \epsilon_i$.  Therefore, rows $j$ and $i$ of $(\lambda, \epsilon)$
  are identical.  By swapping rows $i$ and $j$ of $\delta$ we obtain a
  new marking of $(\lambda, \epsilon)$ that is minimal and agrees with
  $\mu$ for all rows $k \leq i$.  The result follows by induction.
  
  For the general case, let $x \in \orb_{\lambda, \epsilon}$ and let
  $\catB = \{v_{i,j}\}$ be a colored Jordan basis for $x$ of type
  $(\lambda, \epsilon)$.  If we write $v = \sum v_{i, \mu_i}$ and $w =
  \sum v_{i, \delta_i}$ then there is an element $k \in K$ such that
  $k \cdot x = x$ and $k v = w$.  For each $m \in \zmod$, write
  \begin{align*}
    \mu^m_i &= \begin{cases} \mu_i & \epsilon_i + [\lambda_i - \mu_i]
      = m \\ 0 & \text{otherwise;} \end{cases} \\
    \delta_i^m &= \begin{cases} \delta_i & \epsilon_i + [\lambda_i -
      \delta_i] = m \\ 0 & \text{otherwise.} \end{cases}
  \end{align*}
  Set $v_m = \sum v_{i, \mu_i^m}$ and $w_m = \sum v_{i, \delta_i^m}$.
  Then $v = \sum v_m$ and $w = \sum w_m$.  It is evident that $k v =
  w$, so $(v, x)$ and $(w, x)$ lie in the same orbit in $\kvnil$.  But
  $\mu^m$ and $\delta^m$ are minimal by (4), hence $(\lambda,
  \epsilon, \mu^m)$ and $(\lambda, \epsilon, \delta^m)$ must be
  equivalent by (3).  This shows that we need only reorder the rows
  color by color to get the result we desire.
\end{proof}

Let $\ds \catP_\n^\text m$ denote the set of equivalence classes of
marked $\n$-colored partitions.  We define a binary operation $\ds
\cup: \catP_\n^\text m \times \catP_\n^\text m \to \catP_\n^\text m$
as follows.  Let $(\lambda, \epsilon, \mu)$ and $(\alpha, \beta,
\gamma)$ be representatives of elements of $\ds \catP_\n^\text m$.  We
can define $(\lambda, \epsilon, \mu) \cup (\alpha, \beta, \gamma)$ to
be the equivalence class of $(\Lambda(\lambda, \alpha),
\Lambda(\epsilon, \beta), \Lambda(\mu, \gamma))$, where \[\Lambda(f,
g)(i) = \begin{cases}
  f(i / 2) & i \text{ even,} \\
  g((i + 1) / 2) & i \text{ odd.}
\end{cases}\] In other words, we interlace the rows of the two objects
and then permute them to form a colored partition.

The operation $\cup$ is well-defined on equivalence classes and
defines an Abelian monoid structure on $\ds \catP_\n^\text m$.  What
is more, it is evident that the signature is a monoid homomorphism:
\[\xi((\lambda, \epsilon, \mu) \cup (\alpha, \beta, \gamma)) =
\xi(\lambda, \epsilon, \mu) + \xi(\alpha, \beta, \gamma).\] The set
$\catP_\n$ of $\n$-colored partitions is naturally a submonoid of $\ds
\catP_\n^\text m$ via the embedding $(\lambda, \epsilon) \mapsto
(\lambda, \epsilon, 0)$.  Also, if $k$ is a divisor of $\n$ then $\ds
\phi_k: \catP_\n^\text m \to \catP_k^\text m$ is a surjective monoid
homomorphism.
   
If $\mu$ is a minimal marking of $\lambda$ as given in the theorem
then there is a well-defined way of decomposing $(\lambda, \epsilon,
\mu)$ by selecting exactly those rows with nonzero marking.  Let $A =
\{i \in \setN \mid \mu_i > 0\}$ and $B = \{ i \in \setN \mid \lambda_i
> 0, \mu_i = 0 \}$.  Then \[(\lambda, \epsilon, \mu) =
\Delta_B(\lambda, \epsilon, \mu) \cup \Delta_A(\lambda, \epsilon,
\mu).\]

We call $\Delta_B(\lambda, \epsilon, \mu)$ the \emph{characteristic}
\gqb of $(\lambda, \epsilon, \gamma)$.  If we set $(\alpha, \beta,
\gamma) = \Delta_B(\lambda, \epsilon, \mu)$ then
\begin{enumerate}
\item $\gamma_i > 0$ for each $1 \leq i \leq l(\alpha)$ ;
\item $\gamma_i > \gamma_j$ and $\alpha_i - \gamma_i > \alpha_j -
  \gamma_j$ for each $(i, j)$ satisfying $1 \leq i < j \leq l(\alpha)$
  and $\beta_i + [\alpha_i - \gamma_i] = \beta_j + [\alpha_j -
  \gamma_j]$.
\end{enumerate}
If $\beta_i + [\alpha_i - \gamma_i] = \beta_j + [\alpha_j - \gamma_j]$
for each $1 \leq i < j \leq l(\alpha)$ then we simply call
$\Delta_B(\lambda, \epsilon, \mu)$ a \emph{characteristic
  bipartition}.
  
On the other hand, if we set $(\alpha, \beta, \gamma) =
\Delta_A(\lambda, \epsilon, \mu)$ then $\gamma_i = 0$ for each $i$.
So, we have the following result:

\begin{corollary}
  The product $\cup$ defines a bijection onto the image of $\Psi$ from
  the set of pairs $((\lambda, \epsilon, \mu), (\alpha, \beta)) \in
  \catP_\n^\text m \times \catP_\n$ that satisfy
  \begin{enumerate}
  \item $\xi(\lambda, \epsilon) + \xi(\alpha, \beta) = \xi(V)$;
  \item $(\lambda, \epsilon, \mu)$ is a characteristic \gqb.
  \end{enumerate}
\end{corollary}
  
\begin{corollary}
  $\kvnil$ is in bijection with the set of pairs $((\lambda, \epsilon,
  \mu), (\alpha, \beta)) \in \catP_\n^\text m \times \catP_\n$ with
  $\xi(\lambda, \epsilon) + \xi(\alpha, \beta) = \xi(V)$ and
  $(\lambda, \epsilon, \mu)$ a characteristic colored bipartition.
\end{corollary}

\section{The dimension of an orbit} \label{sec:dim}

In this section we construct elementary formulas for the dimension of
an orbit in $\knil$ or $\kvnil$.  This enables us to easily compute
the dimension of an orbit directly from a corresponding combinatorial
parameter (colored partition or \cqb).  We begin by presenting a few
examples that are well known.  We then construct a single formula that
has each of these examples as a special case.  As a consequence, we
will obtain a simple formula for the enhanced signed case $\n = 2$.

\subsection{Known examples} \label{sec:examples}

By way of comparison, we present a few relevant examples from
classical theory.  We begin with a convenient formula.  If $\lambda$
is a partition and $\lambda^t$ its transpose then we define \[
\eta(\lambda) = \sum_{i=1}^{l(\lambda)} (i-1) \lambda_i =
\sum_{i=1}^{l(\lambda)} \binom{\lambda_i^t} 2.  \]

It is well known that $G \cong GL(V)$ acts on the set of nilpotent
endomorphisms of $\ds V$ by conjugation.  In our formulation, this is
the case $\n = 1$.  The orbits are parametrized by partitions
$\lambda$ of size $k = \dim V$ and the dimension of the orbit
corresponding to $\lambda$ is given by \[ \dim \orb_\lambda = 2 \binom
k 2 - 2\eta(\lambda) = k^2 - \sum_{i = 1}^{l(\lambda)}
(\lambda_i^t)^2.  \]

We discussed earlier that if $\n = 2$ then $K \backslash \nil$ is
parametrized by signed partitions of signature $(\dim V_0, \dim V_1)$,
hence of size $k = \dim V_0 + \dim V_1 = \dim V$.  From classical
theory we know that the dimension of the orbit corresponding to
$(\lambda, \epsilon)$ is given by \[ \dim \orb_{\lambda, \epsilon} =
\binom k 2 - \eta(\lambda) = \frac 12 \dim \orb_\lambda = \frac 1 2
\dim \phi_1(\orb_{\lambda, \epsilon}).  \]

Lastly, we mention the formula given in \cite{AH} ($\n = 1$, once
again).  If $G = GL(V)$ acts on $V \times \catN$ by conjugation (where
here $\nil$ includes all nilpotent elements of $\End(V)$) then orbits
are parametrized by bipartitions $(\mu; \nu)$, where $\lambda = \mu +
\nu$ is any partition of size $k = \dim V$.  The dimension of an orbit
$\orb_{\mu;\nu} \in G \backslash (V \times \nil)$ is
\[ \dim \orb_{\mu; \nu} = \dim \orb_\lambda + |\mu| = \dim
\orb_\lambda = 2 \binom k 2 - 2\eta(\lambda) + |\mu|.  \]

\subsection{The dimension formula}

In the signed case one might guess, by analogy with the examples given
above, that if $(\lambda, \epsilon, \mu)$ is a \cqb[2] (or perhaps a
related signed bipartition) then $\dim \orb_{\lambda, \epsilon, \mu} =
\frac 12 \dim \orb_{\lambda} + \frac 12 |\mu| = \frac 12 \orb_{\mu,
  \lambda - \mu}$.  It is obvious from the outset, however, that this
would be overly optimistic as there is no guarantee that this is even
an integer.  We will see, however, that the correct formula is as
close to our guess as could reasonably be hoped.

We once again find the Achar-Henderson strategy to be an excellent
model for proving the general case.  The following definitions and
lemmas are entirely analogous to theirs.  We just need to make a few
minor changes to adapt them to our needs.

\begin{definition}
  For fixed $(v, x) \in V \times \nil$ we define the following
  auxiliary sets:
  \begin{align*}
    E^x &= \{ y \in \End (V) \mid yx = xy \}, \\
    E^{v,x} &= \{ y \in E^x \mid y \cdot v = 0\}, \\
    F^x &= \{ y \in E^x \mid y (V_i) \subset V_i \}, \\
    F^{v, x} &= \{ y \in F^x \mid y \cdot v = 0 \} = E^{v, x} \cap
    F^x, \\
    K^x &= F^x \cap K = E^x \cap K, \\
    K^{v, x} &= \{ y \in F^x \mid y \cdot v = v \}.
  \end{align*}
  Note that $E^x$, $E^{v,x}$, $F^x$, and $F^{v,x}$ are all linear
  spaces and that $K^x$ and $K^{v,x}$ are subgroups of $K$.
\end{definition}

\begin{proposition} \label{prop:dim} If $(v, x) \in V \times \nil$
  then $K^x$ and $K^{v, x}$ are connected algebraic groups and
  \begin{align*}
    \dim \orb_x &= \dim K - \dim F^x, \\
    \dim \orb_{v, x} &= \dim K - \dim F^x + \dim F^xv.
  \end{align*}
\end{proposition}

\begin{proof}  
  $K$ acts transitively on $\orb_x$, so $\orb_x \cong K / K^x$, hence
  $\dim \orb_x = \dim K - \dim K^x$.  Now, $K^x$ is the principal open
  subvariety of (clearly connected) $F^x$ determined by $\det$, so
  $K^x$ is connected and $\dim K^x = \dim F^x$.  Therefore, $\dim
  \orb_x = \dim K - \dim F^x$.
    
  Similarly, $\orb_{v, x} \cong K / K^{v, x}$, hence $\dim \orb_{v, x}
  = \dim K - \dim K^{v, x}$.  $K^{v, x}$ is the principal open
  subvariety of $\{ y \in F^x \mid y \cdot v = v \}$ (which is
  isomorphic to $F^{v, x}$ via the map $y \mapsto y - 1$) determined
  by $\det$.  Therefore, $K^{v, x}$ is connected and $\dim K^{v, x} =
  \dim F^{v,x}$, so $\dim \orb_{v, x} = \dim K - \dim F^{v, x}$.
  Lastly, the multiplication map $F^x \to F^xv$ defined by $y \mapsto
  y \cdot v$ is linear and surjective, with kernel equal to $F^{v,
    x}$.  By the rank-nullity theorem, $\dim F^x v + \dim F^{v,x} =
  \dim F^x$, so $\dim \orb_{v, x} = \dim K - \dim F^x + \dim F^x v$.
\end{proof}

\begin{proposition}
  Fix $x \in \nil$ and let $\catB = \{ v_{i,j} \}$ be a colored Jordan
  basis for $x$ of type $(\lambda, \epsilon)$.  For $k, a, b \in
  \setN$ satisfying $1 \leq a \leq l(\lambda)$ and $1 \leq b \leq
  \lambda_a\}$ let $y_{k, a, b}$ denote the linear endomorphism of $V$
  defined by $y_{k, a, b}(v_{k, j}) = v_{a, b + j - \lambda_k}$ and
  $y_{k, a, b}(v_{i,j}) = 0$ if $i \neq k$.  Then
  \begin{enumerate}
  \item $E^x$ has basis $\catB_E = \{y_{k, a, b} \mid 1 \leq k, a \leq
    l(\lambda) , 1 \leq b \leq \min\{\lambda_a, \lambda_k\} \}$, so
    \begin{align*}
      \dim E^x &= \sum_{k  =  1}^{l(\lambda)} | s_{\lambda_k} (x) | \\
      &= \dim V + 2 \eta(\lambda).
    \end{align*}
        
  \item $F^x$ has basis $\catB_F = \{y_{k,a,b} \in \catB_E \mid
    \epsilon_a + [\lambda_a - b] = \epsilon_k \}$, so
    \[\dim F^x = \sum_{k = 1}^{l(\lambda)}
    s_{\lambda_k}(x)(\epsilon_k).\]
  \end{enumerate}
\end{proposition}

\begin{proof} \
  \begin{enumerate}
  \item If $y \in E^x$ then $y v_{i,j} = y v_{i, \lambda_i -
      (\lambda_i - j)} = y x^{\lambda_i - j} v_{i, \lambda_i} =
    x^{\lambda_i - j} y v_{i, \lambda_i}$, so $y$ is determined by the
    values of $y v_{k, \lambda_k}$.  Write \[y v_{k, \lambda_k} =
    \sum_{i,j} a_{i,j} v_{i,j} = \sum_{i,j} a_{i,j} y_{k,
      i,j}(v_{k,\lambda_k}),\] so the span of the set of $y_{k,a,b}$
    certainly contains $E^x$.  That this set is linear independent
    follows from basic linear algebra.  Therefore, \[0 = y \cdot 0 = y
    x^{\lambda_k}v_{k, \lambda_k} = x^{\lambda_k} y v_{k, \lambda_k} =
    \sum_{i,j} a_{i,j} v_{i,j - \lambda_k}.\] By linear independence,
    if $a_{i,j} \neq 0$ then $j - \lambda_k \leq 0$, hence $j \leq
    \lambda_k$.  We conclude that $E^x$ is contained in the span of
    $\catB_E$.  It is easy to verify, however, that each element of
    $\catB_E$ lies in $E^x$.
    
    It is clear, then, that
    \begin{align*}
      \dim E^x &= \sum_{k = 1}^{l(\lambda)} \sum_{a = 1}^{l(\lambda)}
      \# \{b
      \mid 1 \leq b \leq \min\{\lambda_a, \lambda_k\}\} \\
      &= \sum_{k = 1}^{l(\lambda)} \sum_{a = 1}^k \# \{b \mid 1 \leq b
      \leq \lambda_k\} + \sum_{k =1}^{l(\lambda)} \sum_{a = k +
        1}^{l(\lambda)} \# \{b
      \mid 1 \leq b \leq \lambda_a\} \\
      &= \sum_{k = 1}^{l(\lambda)} \sum_{a = 1}^k \lambda_k + \sum_{k
        =1}^{l(\lambda)} \sum_{a = k + 1}^{l(\lambda)} \lambda_a
      \\
      &= \sum_{k = 1}^{l(\lambda)} k \lambda_k + \sum_{a =
        1}^{l(\lambda)} \sum_{k = 1}^{a - 1} \lambda_a \\
      &= \sum_{k = 1}^{l(\lambda)} \lambda_k + \sum_{k =
        1}^{l(\lambda)} (k - 1) \lambda_k + \sum_{a = 1}^{l(\lambda)}
      (a - 1) \lambda_a \\
      &= \dim V + 2 \eta(\lambda).
    \end{align*}

    The other formula for $\dim E^x$ follows from the fact that
    $E^xv_{k, \lambda_k} = \ker x^{\lambda_k}$.
    
  \item If $y_{k, a, b} \in F^x$ then $\epsilon_k = \chi(v_{k,
      \lambda_k}) = \chi (y_{k, a, b} v_{k, \lambda_k}) = \chi(v_{a,
      b}) = \epsilon_a + [\lambda_a - b]$.  We already know that such
    elements of $\catB_E$ are linearly independent and it is a quick
    exercise to verify that they are in $F^x$.  The dimension formula
    should be clear once we observe that for fixed $k$ the set
    $\{y_{k, a, b} v_{k, \lambda_k}\}$ is a basis for $\ker
    x^{\lambda_k} \cap V_{\epsilon_k}$.  \qedhere

  \end{enumerate}
\end{proof}

\begin{proposition} \label{prop:exv} Let $(v, x) \in V \times \nil$.
  Let $\catB = \{v_{i,j}\}$ be a colored Jordan basis for $x$ of type
  $(\lambda, \epsilon)$ and write $v = \sum a_{i,j} v_{i,j}$.  For
  convenience, set $a_{i,j} = 1$ if $j < 1$.  For each $m \in \zmod$
  we define a marking of $\lambda$: $\mu^m_i = \max \{j \in \setZ \mid
  a_{i,j} \neq 0, \xi(v_{i,j}) = m\}$.  We also define $\mu_i = \max
  \{j \in \setZ \mid a_{i,j} \neq 0\} = \max\{\mu^m_i \mid m \in
  \zmod\}$.  Let the corresponding minimal bipartitions be $(\lambda,
  \tilde \mu^m)$ and $(\lambda, \tilde \mu)$.  Then
  \begin{enumerate}
  \item $\ds \catB^{\tilde \mu}$ is a colored Jordan basis for $E^xv$.
    In particular, $E^xv$ is colored and $x$-stable, with $\ds
    \xi(E^xv) = \xi(\mu, \epsilon + [\lambda - \nu])$, so $\dim E^xv =
    |\tilde \mu|$.
  \item $\ds \bigsqcup_{m} \left(\catB^{\tilde \mu^m} \cap V_m
    \right)$ is a colored basis for $F^xv$.  In particular, $F^xv$ is
    colored and $x^{\n}$-stable, with $\ds \xi_m(F^xv) = \sum_{i =
      1}^{l(\lambda)} \left\lceil \frac {\mu^m_i} \n \right \rceil$,
    so $\ds \dim (F^x v) = \sum_{m = 0}^{\n - 1} \sum_{i =
      1}^{l(\lambda)} \left\lceil \frac {\mu^m_i} \n \right \rceil$.
  \end{enumerate}
\end{proposition}

\begin{proof}
  The proof of (2) should be clear once we have proved (1).  Since $x
  \in E^x$ it is clear that $E^xv$ is $x$-stable.  Now, $y_{k, k,
    \lambda_k} v = \sum a_{i,j} y_{k, k, \lambda_k} v_{i,j} = \sum
  a_{k,j} v_{k,j}$.  Set $v_k = \sum a_{k,j} v_{k,j}$.  It is clear,
  then, that $E^xv = E^x v_1 + \dots + E^xv_{l(\lambda)}$.  So, we may
  assume that $v= v_k$ lives in a single Jordan block.
  
  Since $E^xv$ is a vector space, we may assume that $a_{k, \mu_k} =
  1$.  Now, $y = y_{k,k,\lambda_k} - a_{k, \mu_{k - 1}} y_{k,k,
    \lambda_k - 1}$ is in $E^x$.  But $yv$ has no $v_{k,
    \mu_{k-1}}$-component.  By a similar construction, we may
  successively eliminate each component of $v_k$, leaving $v_{k,
    \mu_k}$.  In other words, we have shown that $v_{k, \mu_k} \in
  E^xv$.  But then by $x$-stability we have $v_{k, j} \in E^xv$ for
  each $1 \leq j \leq \mu_k$.  This also shows that some subset of
  $\catB$ is a basis of $E^xv$.
  
  Now, suppose that $v_{i,j} \in E^x v$, with $j > \mu_i$.  This
  occurs precisely if there is a $k \neq i$ with a choice of $a, b$
  such that $v_{i,j} = y_{k, a, b} v_{k, \mu_k} = v_{a, b + \mu_k -
    \lambda_k}$ and $1 \leq b \leq \min\{\lambda_a, \lambda_k\}$.
  Obviously, we must have $a = i$ and $j = b + \mu_k - \lambda_k$,
  with $1 \leq b \leq \min\{\lambda_i, \lambda_k\}$.  Substituting, we
  have $1 \leq j + \lambda_k - \mu_k \leq \min\{\lambda_i,
  \lambda_k\}$.  If $k < i$ then we have $j + \lambda_k - \mu_k \leq
  \lambda_i$, or $j \leq \lambda_i - (\lambda_k - \mu_k)$.  If $k > i$
  then we have $j + \lambda_k - \mu_k \leq \lambda_k$, or $j \leq
  \mu_k$.  Therefore, $v_{i,j} \in E^xv$ if and only if $j \leq
  \max(\{\mu_k \mid k \geq i\} \cup \{ \lambda_i - (\lambda_k - \mu_k)
  \mid k \leq i \})$.  In other words, $j \leq \tilde \mu_i$.
  
  The remainder of the claims follow immediately.
\end{proof}

We pause here to observe that propositions \ref{prop:exv} and
\ref{prop:rhobar} give an alternate proof that the \cqb associated to
$\orb$ is unique.  Proposition \ref{prop:exv} gives a canonical
interpretation of $(\lambda, \epsilon, \tilde \mu)$ that shows it is
an orbit invariant.  Proposition \ref{prop:rhobar} shows that any \cqb
corresponding to the orbit must be equal to $\bar \rho (\lambda,
\epsilon, \tilde \mu)$, hence is completely determined.  Similarly, if
$(\lambda, \epsilon, \mu)$ is a \cqb corresponding to $(v, x)$ and $W
= \setF[x](F^x(v))$ then $x|_W$ has colored Jordan type $(\mu,
\epsilon + [\lambda - \mu])$.

\begin{corollary}
  If $(v, x) \in \vnil$ corresponds to the \cqb $(\lambda, \epsilon,
  \mu)$ then $\ds \dim F^x v = \sum_{i = 1}^{l(\lambda)} \left\lceil
    \frac {\mu_i} \n \right \rceil$.
\end{corollary}

\begin{corollary}
  Let $(\lambda, \epsilon, \mu)$ be a \cqb with $(\lambda, \epsilon,
  \tilde \mu) = \bar \rho (\lambda, \epsilon, \mu)$ and set $\tilde
  \nu = \lambda - \tilde \mu$ and $\tilde \epsilon = \epsilon +
  [\nu]$.  If $(v, x) \in \vnil$ then $(v, x) \in \orb_{\lambda,
    \epsilon, \mu}$ if and only if $x |_{E^xv}$ has colored Jordan
  type $(\tilde \mu, \tilde \epsilon)$ and $x |_{V / E^xv}$ has
  colored Jordan type $(\tilde \nu, \epsilon)$.
\end{corollary}

\begin{proof}
  The proposition, plus lemma \ref{lem:sum}, tells us the colored
  Jordan type of $x|_{E^xv}$ and of $x|_{V / E^xv}$.  Conversely, if
  $x|_{E^xv}$ and of $x|_{V / E^xv}$ are determined, there is only one
  way to pair them to get a colored bipartition, so the \cqb is
  determined, as well.
\end{proof}

\begin{corollary} \label{cor:dimA} If $(\lambda, \epsilon, \mu)$ is a
  \cqb and $s$ is as given in definition \ref{def:s} then
  \begin{align*}
    \dim \orb_x &= \sum_i \left(\dim V_i\right)^2 - \sum_{k =
      1}^{l(\lambda)}
    s_{\lambda_k}(x)(\epsilon_k), \\
    \dim \orb_{v, x} &= \sum_i \left(\dim V_i\right)^2 - \sum_{k =
      1}^{l(\lambda)} s_{\lambda_k}(x)(\epsilon_k) + \sum_{i =
      1}^{l(\lambda)} \left\lceil \frac {\mu_i} \n \right \rceil.
  \end{align*}
\end{corollary}

\begin{corollary} \label{cor:dimB} If $\n = 1$ and $(\lambda,
  \epsilon, \mu)$ is a \cqb[1] (bipartition) then
  \begin{align*}
    \dim \orb_{\lambda, \epsilon} & =  2 \binom {\dim V} 2 - 2\eta(\lambda), \\
    \dim \orb_{\lambda, \epsilon, \mu} & = 2 \binom {\dim V} 2 -
    2\eta(\lambda) + |\mu|.
  \end{align*}
\end{corollary}

\begin{proof}
  If $\n = 1$ then $F^x = E^x$.
\end{proof}

Once again, we recall that if $\n = 2$ then we customarily use $+$ and
$-$ in place of $0$ and $1$, respectively, as the colors that decorate
our partitions.  So, by a signed \qb[2] of signature $(p, q)$ we
simply mean a \cqb[2] that has $p$ boxes labeled with $+$ and $q$
boxes labeled with $-$.

\begin{corollary} \label{cor:dim} If $\n = 2$ then orbits in $\kvnil$
  are parametrized by signed \qb[2]{}s.  If $(\lambda, \epsilon, \mu)$
  is a signed \qb[2] then
  \begin{align*}
    \dim \orb_{\lambda, \epsilon} & =  \binom {\dim V} 2 - \eta(\lambda), \\
    \dim \orb_{\lambda, \epsilon, \mu} & = \binom {\dim V} 2 -
    \eta(\lambda) + \sum_{i=1}^{l(\lambda)} \ceil {\frac {\mu_i} 2}.
  \end{align*}
\end{corollary}

\end{document}